\documentclass{article}
\usepackage{amsmath}
\usepackage{amsfonts}

\newtheorem{theorem}{Theorem}
\newtheorem{definition}{Definition}
\newtheorem{lemma}{Lemma}

\newcommand{\sign}{{\rm sign}\hskip0.02cm}
\newcommand{\supp}{{\rm supp}\hskip0.02cm}
\newcommand{\vol}{{\rm vol}\hskip0.02cm}

\newcommand{\lin}{{\rm lin}\hskip0.02cm}

\newcommand{\ep}{\varepsilon}
\newcommand{\mt}{{\cal T}}

\newcommand{\ma}{\mathfrak{a}}
\newcommand{\mn}{\mathfrak{N}}
\newcommand{\cm}{{\check\mu}}
\newcommand{\co}{{\check\Omega}}

\newenvironment{remark}[1]{\medskip\par\noindent{\bf #1.}\rm}
{\medskip\par\noindent}

\newenvironment{proof}[1]{\medskip\par\noindent{\sc #1.\ }}
{~\rule{0.5em}{0.5em}\medskip\par}

\begin{document}

\title{\LARGE{\bf{
Sufficient enlargements of minimal volume for finite dimensional
normed linear spaces}}}

\author{{\sc M.I.~Ostrovskii}\\
\\
Department of Mathematics and Computer Science\\
St. John's University\\
8000 Utopia Parkway\\
Queens, NY 11439, USA\\
e-mail: {\tt ostrovsm@stjohns.edu}\\
Phone: (718)-990-2469\\
Fax: (718)-990-1650}

\date{\today}
\maketitle

\begin{large}

\noindent{\bf Abstract.} Let $B_Y$ denote the unit ball of a
normed linear space $Y$. A symmetric, bounded, closed, convex set
$A$ in a finite dimensional normed linear space $X$ is called a
{\it sufficient enlargement} for $X$ if, for an arbitrary
isometric embedding of $X$ into a Banach space $Y$, there exists a
linear projection $P:Y\to X$ such that $P(B_Y)\subset A$.  The
main results of the paper: {\bf (1)} Each minimal-volume
sufficient enlargement is linearly equivalent to a zonotope
spanned by multiples of columns of a totally unimodular matrix.
{\bf (2)} If a finite dimensional normed linear space has a
minimal-volume sufficient enlargement which is not a
parallelepiped, then it contains a two-dimensional subspace whose
unit ball is linearly equivalent to a regular hexagon. \vskip1cm


\section{Introduction}

This paper is devoted to a generalization of the main results of
\cite{Ost04}, where similar results were proved in the dimension
two. We refer to \cite{Ost04,Ost07+} for more background and
motivation.

\subsection{Notation and definitions}

All linear spaces considered in this paper will be over
the reals.
By a {\it space} we mean a normed linear space,
unless it is explicitly mentioned otherwise.
We denote by $B_X$ ($S_X$)
the unit ball (sphere) of a space $X$. We say that
subsets $A$ and $B$ of finite dimensional linear spaces $X$ and $Y$,
respectively, are {\it linearly equivalent} if there exists a
linear isomorphism $T$ between the subspace spanned by $A$ in
$X$ and the subspace spanned by $B$ in $Y$ such that $T(A)=B$.
By a {\it symmetric} set $K$ in a linear space we mean a
set such that $x\in K$ implies $-x\in K$.
\medskip

Our terminology and notation of Banach space theory follows
\cite{JL01}. By $B_p^n$, $1\le p\le\infty$, $n\in \mathbb{N}$ we
denote the closed unit ball of $\ell_p^n$. Our terminology and
notation of convex geometry follows \cite{Sch93}.
\medskip

We use the term {\it ball}~ for a
symmetric, bounded, closed, convex set with interior points
in a finite dimensional linear space.

\begin{definition} {\rm\cite{extracta} A ball $A$ in a finite dimensional normed
space $X$ is called a {\it sufficient enlargement} (SE) for $X$
(or of $B_X$) if, for an arbitrary isometric embedding of $X$ into
a Banach space $Y$, there exists a projection $P:Y\to X$ such that
$P(B_Y)\subset A$. A sufficient enlargement $A$ for $X$ is called
a {\it minimal-volume sufficient enlargement} (MVSE) if $\vol
A\le\vol D$ for each SE $D$ for $X$.}
\end{definition}

It can be proved, using a standard compactness
argument and Lemma \ref{L:H} below,
that minimal-volume sufficient enlargements exist for every
finite dimensional space.
\medskip

Recall that a real matrix $A$ with entries $-1$, $0$, and $1$ is
called {\it totally unimodular} if all minors (that is,
determinants of square submatrices) of $A$ are equal to $-1, 0$,
or $1$. See \cite{padberg} and \cite[Chapters~19--21]{Sch86} for a
survey of results on totally unimodular matrices and their
applications.
\medskip

A Minkowski sum of finitely many line segments in a linear space
is called a {\it zonotope} (see
\cite{bolker,martini,McM71,Sch93,schneiderweil} for basic facts on
zonotopes). We consider zonotopes that are sums of line segments
of the form $I(x)=\{\lambda x:~ -1\le\lambda\le 1\}$. For a
$d\times m$ totally unimodular matrix with columns $\tau_i$
$(i=1,\dots,m)$ and real numbers $a_i$ we consider the zonotope
$Z$ in $\mathbb{R}^d$ given by
$$Z=\sum_{i=1}^mI(a_i\tau_i).$$
The set of all zonotopes that are linearly equivalent to zonotopes
obtained in this way over all possible choices of $m$, of a rank
$d$ totally unimodular $d\times m$ matrix, and of positive numbers
$a_i~ (i=1,\dots,m)$ will be denoted by ${\cal T}_d$. Observe that
each element of ${\cal T}_d$ is $d$-dimensional in the sense that
it spans a $d$-dimensional subspace. It is easy to describe all
$2\times m$ totally unimodular matrices and to show that ${\cal
T}_2$ is the union of the set of all symmetric hexagons and the
set of all symmetric parallelograms.
\medskip

The class $\mt_d$ of zonotopes has been characterized in several
different ways, see \cite{Cox,Erd99,J,McM75,laa,She}. We shall use
a characterization of $\mt_d$ in terms of lattice tiles. Recall
that a compact set $K\subset \mathbb{R}^d$ is called a {\it
lattice tile} if there exists a basis $\{x_i\}_{i=1}^d$ in
$\mathbb{R}^d$ such that
$$\mathbb{R}^d=\bigcup_{m_1,\dots,m_d\in\mathbb{Z}}\left(\left(
\sum_{i=1}^dm_ix_i\right)+K\right),$$ and the interiors of the
sets $(\sum_{i=1}^dm_ix_i)+K$ are disjoint. The set
$$\Lambda=\left\{\sum_{i=1}^dm_ix_i:~m_1,\dots,m_d\in\mathbb{Z}\right\}$$ is
called a {\it lattice}. The absolute value of the determinant of
the matrix whose columns are the coordinates of $\{x_i\}_{i=1}^d$
is called the {\it determinant} of $\Lambda$ and is denoted
$d(\Lambda)$, see \cite[\S~3]{GL87}.
\medskip

\begin{theorem}\label{T:tiles} {\rm\cite{McM75}, \cite{Erd99}} A $d$-dimensional zonotope is a
lattice tile if and only if it is in $\mt_d$.
\end{theorem}

It is worth mentioning that lattice tiles in $\mathbb{R}^d$ do not
have to be zonotopes, see \cite{V,McM80,McM81}, and \cite[Chapter
3]{Z}.

\subsection{Statements of the main results}

The main result of \cite{laa} can be restated in the following
way. (A finite dimensional normed space is called {\it polyhedral}
if its unit ball is a polytope.)

\begin{theorem}\label{T:laa+} A ball $Z$ is linearly equivalent to
an MVSE for some d-dimensional polyhedral space
$X$ if and only if $Z\in{\cal T}_d$.
\end{theorem}

In \cite{Ost04} it was shown that for $d=2$ the statement of
Theorem \ref{T:laa+} is valid without the restriction of
polyhedrality of $X$. The main purpose of the present paper is to
prove the same for each $d\in\mathbb{N}$. It is clear that it is
enough to prove

\begin{theorem}\label{T:MVSE}
Each MVSE for a $d$-dimensional space is in ${\cal T}_d$.
\end{theorem}

Using Theorem \ref{T:MVSE} we show that spaces having
non-parallelepipedal MVSE cannot be strictly convex or smooth.
More precisely, we prove

\begin{theorem}\label{T:NP} Let $X$ be a finite dimensional
normed linear space having an MVSE that is not a parallelepiped.
Then $X$ contains a two-dimensional subspace whose unit ball is
linearly equivalent to the regular hexagon.
\end{theorem}

\begin{remark}{Remarks. 1} Theorem \ref{T:NP} is a simultaneous
generalization of \cite[Theorem 4]{Ost04} (which is a special case
of Theorem \ref{T:NP} corresponding to the case $\dim X=2$) and of
\cite[Theorem 7]{archiv} (which states that each MVSE for
$\ell_2^n$ is a cube circumscribed about $B_2^n$).
\smallskip

\noindent{\bf 2.} The fact that $X$ contains a two-dimensional
subspace whose unit ball is linearly equivalent to a regular
hexagon does not imply that $X$ has an MVSE that is not a
parallelepiped. A simplest example supporting this statement is
$\ell_\infty^3$.
\end{remark}

\section{Proof of Theorem \ref{T:MVSE}}

First we show that it is enough to prove the following lemmas. It
is worth mentioning that our proof of Theorem \ref{T:MVSE} goes
along the same lines as the proof of its two-dimensional version
in \cite{Ost04}. The most difficult part of the proof is a
$d$-dimensional version of the approximation lemma (\cite[Lemma 2,
p.~380]{Ost04}), it is the contents of Lemma \ref{L:APPROX} of the
present paper. Also, a two-dimensional analogue of Lemma
\ref{L:closed} is completely trivial.

\begin{lemma}\label{L:closed}
Let $T_n\subset \mathbb{R}^d,~ n\in\mathbb{N}$ be such that
$T_n\in {\cal T}_d$, and $\{T_n\}_{n=1}^\infty$ converges with
respect to the Hausdorff metric to a $d$-dimensional set $T$. Then
$T\in {\cal T}_d$.
\end{lemma}

\begin{remark}{Remark} If a sequence $\{T_n\}_{n=1}^\infty\subset\mt_d$ converges to a lower-dimensional set
$T$, the set $T$ does not have to be in $\mt_{\dim T}$. In fact,
as it was already mentioned, ${\cal T}_2$ is the set of all
symmetric hexagons and parallelograms. On the other hand, it is
easy to find a Hausdorff convergent sequence of elements of ${\cal
T}_3$ whose limit is an octagon.
\end{remark}

\begin{lemma}[Main lemma]\label{L:APPROX} For each $d\in\mathbb{N}$ there exist
$\psi_d>0$ and a function $\mathfrak{t}_d:(0,\psi_d)\to(1,\infty)$
satisfying the conditions:
\medskip

\noindent{\rm (1)} $\lim_{\varepsilon\downarrow
0}\mathfrak{t}_d(\varepsilon)=1$;
\medskip

\noindent{\rm (2)} If $Y$ is a $d$-dimensional polyhedral space,
$B$ is an MVSE for $Y$, and $A$ is an SE for $Y$ satisfying
\begin{equation}\label{E:lemma1}
\vol A\le (1+\varepsilon)^d\vol B
\end{equation}
for some $0<\ep<\psi_d$, then $A$ contains a ball $\tilde A$
satisfying the conditions:
\medskip

\noindent{\rm (a)} $d(\tilde A,T)\le \mathfrak{t}_d(\varepsilon)$
for some $T\in{\cal T}_d$, where by $d(\tilde A,T)$ we denote the
Banach--Mazur distance;
\medskip

\noindent{\rm (b)} $\tilde A$ is an SE for $Y$.
\end{lemma}

\begin{lemma}\label{L:H} {\rm\cite[Lemma 3]{Ost04}}
The set of all sufficient enlargements for a finite dimensional
normed space $X$ is closed with respect to the Hausdorff metric.
\end{lemma}

\begin{proof}{Proof of Theorem \ref{T:MVSE}} (We assume that
Lemmas \ref{L:closed} and \ref{L:APPROX} have been proved.) Let
$X$ be a $d$-dimensional space and let $A$ be an MVSE for $X$. Let
$\{\varepsilon_n\}_{n=1}^\infty$ be a sequence satisfying
$\psi_d>\varepsilon_n>0$ and $\varepsilon_n\downarrow 0$. Let
$\{Y_n\}_{n=1}^\infty$ be a sequence of polyhedral spaces
satisfying
\begin{equation}\label{E:Y_n}
\frac1{1+\varepsilon_n}B_X\subset
B_{Y_n}\subset B_X.
\end{equation}
Then $A$ is an SE for $Y_n$. Let $B_n$ be an MVSE for $Y_n$. Then
$(1+\varepsilon_n)B_n$ is an SE for $X$. Since $A$ is a
minimal-volume SE for $X$, we have
$$\vol A\le \vol\left((1+\varepsilon_n)B_n\right)=
(1+\varepsilon_n)^d\vol B_n.$$
\smallskip

By Lemma \ref{L:APPROX} for every $n\in\mathbb{N}$ there exists an
SE $\tilde A_n$ for $Y_n$ satisfying
$$\tilde A_n\subset A$$
and
\begin{equation}\label{E:n}
d(\tilde A_n, T_n)\le \mathfrak{t}_d(\varepsilon_n)
\end{equation}
for some $T_n\in{\cal T}_d$.
\medskip

The condition (\ref{E:Y_n}) implies that
$(1+\varepsilon_n)\tilde A_n$ is an SE for $X$.
\medskip

The sequence $\{(1+\varepsilon_n)\tilde A_n\}_{n=1}^\infty$ is
bounded (all of its terms are contained in $(1+\ep_1)A$). By the
Blaschke selection theorem \cite[p.~50]{Sch93} the sequence
$\{(1+\varepsilon_n)\tilde A_n\}_{n=1}^\infty$ contains a
subsequence convergent with respect to the Hausdorff metric. We
denote its limit by $D$, and assume that the sequence
$\{(1+\ep_n)\tilde A_n\}_{n=1}^\infty$ itself converges to $D$.
\medskip

Observe that each $\tilde A_n$ contains $(1/(1+\varepsilon_1))B_X$
and is contained in $A$. By (\ref{E:n}) we may assume without loss
of generality that $T_n$ are balls in $X$ satisfying
\begin{equation}\label{E:inclusions}
\frac1{1+\varepsilon_1}B_X\subset \tilde A_n\subset T_n\subset
\mathfrak{t}_d(\varepsilon_n)\tilde A_n\subset
\mathfrak{t}_d(\varepsilon_n) A.
\end{equation}

It is clear that $D$ is the Hausdorff limit of $\{\tilde
A_n\}_{n=1}^\infty$. From (\ref{E:inclusions}) we get that $D$ is
the Hausdorff limit of $\{T_n\}_{n=1}^\infty$. By Lemma
\ref{L:closed} we get $D\in{\cal T}_d$.
\medskip

By Lemma \ref{L:H} the set $D$ is an SE for $X$. Since
$(1+\varepsilon_n)\tilde A_n\subset (1+\varepsilon_n)A$, and
$(1+\varepsilon_n)A$ is Hausdorff convergent to $A$, we have
$D\subset A$. On the other hand, $A$ is an MVSE for $X$, hence
$D=A$ and $A\in{\cal T}_d$.
\end{proof}

\begin{proof}{Proof of Lemma \ref{L:closed}} By
Theorem \ref{T:tiles} the sets $T_n$ are lattice tiles. Let
$\{\Lambda_n\}_{n=1}^\infty$ be lattices corresponding to these
lattice tiles. Since volume is continuous with respect to the
Hausdorff metric (see \cite[p.~55]{Sch93}), the supremum
$\sup_n\vol (T_n)$ is finite. Since $T_n$ is a lattice tile with
respect to $\Lambda_n$, the determinant of $\Lambda_n$ satisfies
$d(\Lambda_n)=\vol (T_n)$. (Although I have not found this result
in the stated form, it is well known. It can be proved, for
example, using the argument from \cite[pp.~42--43, Proof of
Theorem 2]{GL87}.) Hence $\displaystyle{\sup_n
d(\Lambda_n)<\infty}$. Since $T$ is $d$-dimensional, there exists
$r>0$ such that $rB_2^d\subset T$. Choosing a smaller $r>0$, if
necessary, we may assume that $rB_2^d\subset T_n$ for each $n$.
Therefore the lattices $\{\Lambda_n\}_{n=1}^\infty$ satisfy the
conditions of the selection theorem of Mahler (see, for example,
\cite[\S 17]{GL87}, where the reader can also find the standard
definition of convergence for lattices). Hence the sequence
$\{\Lambda_n\}_{n=1}^\infty$ contains a subsequence which
converges to some lattice $\Lambda$. It is easy to verify that $T$
tiles $\mathbb{R}^d$ with respect to $\Lambda$.

On the other hand, the number of possible distinct columns of a
totally unimodular matrix with columns from $\mathbb{R}^d$ is
bounded from above by $3^d$, because each entry is $0$, $1$, or
$-1$. (Actually a much better exact bound is known, see
\cite[p.~299]{Sch86}.) Using this we can show that $T$ is a
zonotope by a straightforward argument. Also we can use the
argument from \cite[Theorem 3.5.2]{Sch93} and the observation that
a convergent sequence of measures on the sphere of $\ell_2^d$,
each of whom has a finite support of cardinality $\le 3^d$,
converges to a measure supported on $\le 3^d$ points.

Thus, $T$ is a zonotope and a lattice tile. Applying Theorem
\ref{T:tiles} again, we get $T\in {\cal T}_d$.
\end{proof}

\section{Proof of the Main Lemma}

\subsection{Coordinatization}

\begin{proof}{Proof of Lemma \ref{L:APPROX}} In our argument the dimension $d$ is
fixed. Many of the parameters considered below depend on $d$,
although we do not reflect this dependence in our notation.
\medskip

Since $Y$ is polyhedral, we can consider $Y$ as a subspace of
$\ell_\infty^m$. Let $P:\ell_\infty^m\to Y$ be a linear projection
satisfying $P(B_\infty^m)\subset A$ (such a projection exists
because $A$ is an SE). Let $\tilde A=P(B_\infty^m)$. It is easy to
see that $\tilde A$ is an SE for $Y$. It remains to show that
$\tilde A$ is close to some $T\in {\cal T}_d$ with respect to the
Banach--Mazur distance.
\medskip

We consider the standard inner product on $\ell_\infty^m$. (The
unit vector basis is an orthonormal basis with respect to this
inner product.)
\medskip

Let $\{q_1,\dots,q_{m-d}\}$ be an orthonormal basis in $\ker P$.
Let $\{y_1,\dots, y_d\}$ be an orthonormal basis in $Y$. Let
$\tilde q_1,\dots,\tilde q_d$ be such that $\{\tilde
q_1,\dots,\tilde q_d,q_1,\dots,q_{m-d}\}$ is an orthonormal basis
in $\ell_\infty^m$.

\begin{lemma}\label{L:shape} {\bf (Image Shape Lemma)} Let
$P$ and $\tilde q_1,\dots,\tilde q_d$ be as above. Denote by
$\tilde Q=[\tilde q_1,\dots,\tilde q_d]$ the matrix whose columns
are $\tilde q_1, \dots,\tilde q_d$. Let $z_1,\dots,z_m$ be the
columns of the transpose matrix $\tilde Q^T$. Then $P(B_\infty^m)$
is linearly equivalent to the zonotope
$\sum_{i=1}^mI(z_i)\subset\mathbb{R}^d$.
\end{lemma}

\begin{proof}{Proof} It is enough to observe that:
\medskip

\noindent(i) Images of $B_\infty^m$ under
two linear projections with the
same kernel are linearly equivalent. Hence, $P(B_\infty^m)$ is
linearly equivalent to the image of the orthogonal projection
with the kernel $\ker P$.
\medskip

\noindent(ii) The matrix $\tilde Q\tilde Q^T$
is the matrix of the orthogonal projection with the kernel
$\ker P$.
\end{proof}

By Lemma \ref{L:shape} we may replace $\tilde A$ by
\begin{equation}\label{E:Z}
Z=\sum_{i=1}^mI(z_i)
\end{equation}
in the estimate (a) of Lemma \ref{L:APPROX}.
\medskip

Let $M=\binom md$. We denote by $u_i$ $(i=1,\dots,M)$ the $d\times d$
minors of $[y_1,\dots,y_d]$ (ordered in some way). We denote by
$w_i$ $(i=1,\dots,M)$ the $d\times d$
minors of $[\tilde q_1,\dots,\tilde q_d]$ ordered in the same
way as the $u_i$. We denote by
$v_i$ $(i=1,\dots,{\binom m{m-d}}=M)$ their complementary
$(m-d)\times(m-d)$
minors of $[q_1,\dots,q_{m-d}]$. Using the word {\it complementary}
we mean that all minors are considered as minors of the matrix
$[\tilde q_1,\dots,\tilde q_d,q_1,\dots,q_{m-d}]$,
see \cite[p.~76]{aitken}.
\medskip

By the Laplacian expansion (see
\cite[p.~78]{aitken})
$$
\det[y_1,\dots,y_d,q_1,\dots,q_{m-d}]=
\sum_{i=1}^M\theta_iu_iv_i
$$
and
\begin{equation}\label{E:laplace}
\det[\tilde q_1,\dots,\tilde q_d,q_1,\dots,q_{m-d}]=
\sum_{i=1}^M\theta_iw_iv_i
\end{equation}
for proper signs $\theta_i$.
\medskip

Since the matrix $[\tilde q_1,\dots,\tilde q_d,q_1,\dots,q_{m-d}]$
is orthogonal, we have
\begin{equation}\label{E:pm1}
\det[\tilde q_1,\dots,\tilde q_d,q_1,\dots,q_{m-d}]=\pm1.
\end{equation}

We need the following result on compound matrices.
(We refer to \cite[Chapter V]{aitken} for necessary definitions and
background.)
\medskip

{\it A compound matrix of an orthogonal
matrix is orthogonal} (see \cite[Example 4 on p.~94]{aitken}).
\medskip

This result implies, in particular, that the Euclidean norms of
the vectors $\{w_i\}_{i=1}^M$ and $\{v_i\}_{i=1}^M$ in
$\mathbb{R}^M$ are equal to $1$.
\medskip

From (\ref{E:laplace}) and (\ref{E:pm1}) we get that
either
\smallskip

(a) $w_i=\theta_iv_i$ for every $i$
\smallskip

\noindent or
\smallskip

(b) $w_i=-\theta_iv_i$ for every $i$.
\medskip

Without loss of generality, we assume that $w_i=\theta_iv_i$
for all $i$ (we replace $q_1$ by $-q_1$ if it
is not the case).
\medskip

We compute the volume of $\tilde A$ and $B$
with the normalization that
comes from the Euclidean structure introduced above. It is
well known (see \cite[p.~318]{jfa}) and is easy to verify that
with this normalization
$$\vol\tilde A=\frac{2^d}{\left|\sum_{i=1}^M\theta_iu_iv_i\right|}
\sum_{i=1}^M|v_i|$$
and
$$\vol B=\frac{2^d}{\max_i|u_i|}$$
for each MVSE $B$ for $Y$.

\begin{remark}{Remark} After the publication of \cite{jfa}
I learned that the formula for the volume of a zonotope used in
\cite{jfa} can be found in \cite[Appendix, Section VI]{blaschke}.
\end{remark}

Since $\vol\tilde A\le\vol A$, the inequality (\ref{E:lemma1})
implies that
\begin{equation}\label{E:min2}
\max_i|u_i|\sum_{i=1}^M|v_i|\le
(1+\varepsilon)^d
\left|\sum_{i=1}^M\theta_iu_iv_i\right|.
\end{equation}

By (a) the inequality (\ref{E:min2}) can be rewritten as
\begin{equation}\label{E:min3}
\max_i|u_i|\sum_{i=1}^M|w_i|\le
(1+\varepsilon)^d
\left|\sum_{i=1}^Mu_iw_i\right|.
\end{equation}

We need the following two observations:
\medskip

\noindent(i)~ $2^d\sum_{i=1}^M|w_i|$ is the volume of $Z$ in
$\mathbb{R}^d$.
\medskip

\noindent(ii) The vector $\{u_i\}_{i=1}^M$ is what is called the
{\it Grassmann coordinates}, or the {\it Pl\"ucker coordinates} of
the subspace $Y\subset\mathbb{R}^m$, see \cite[Chapter VII]{HP} and
\cite[p.~42]{sh}. Recall that $Y$ is spanned by the columns of the
matrix $[y_1,\dots,y_d]$. It is easy to see that if we choose
another basis in $Y$, the Grassman (Pl\"ucker) coordinates will be
multiplied by a constant.
\medskip

We denote by $\mathcal{Z}_\varepsilon$ ($\ep>0$) the set of all
$d$-dimensional zonotopes in $\mathbb{R}^d$ satisfying the
condition (\ref{E:min3}) with an equality. More precisely, we
define $\mathcal{Z}_\varepsilon$ as the set of those
$d$-dimensional zonotopes $Z$ in $\mathbb{R}^d$ for which

\begin{itemize}
\item[(1)] There exists $m\in\mathbb{N}$ and a rank $d$ matrix
$\tilde Q$ of size $m\times d$ such that
 $Z=\sum_{i=1}^mI(z_i)$, where $z_i\in\mathbb{R}^d$, $i=1,\dots,
 m$, are rows of $\tilde Q$.

\item[(2)] There exists a rank $d$ matrix $Y$ of size $m\times d$
such that, if we denote the $d\times d$ minors of $\tilde Q$ by
$\{w_i\}_{i=1}^\infty$, where $M=\binom{m}d$, and the $d\times d$
minors of $Y$, ordered in the same way as the $w_i$, by
$\{u_i\}_{i=1}^\infty$, then
\begin{equation}\label{E:equal}
\max_i|u_i|\sum_{i=1}^M|w_i|=(1+\varepsilon)^d
\left|\sum_{i=1}^Mu_iw_i\right|,
\end{equation}
and there is no $Y$ for which
$$\max_i|u_i|\sum_{i=1}^M|w_i|<(1+\varepsilon)^d
\left|\sum_{i=1}^Mu_iw_i\right|.$$
\end{itemize}

\begin{remark}{Remarks. 1} It is clear that the zonotope property of being in
$\mathcal{Z}_\ep$ is invariant under changes of the system of
coordinates.
\medskip

\noindent{\bf 2.} We do not consider the class $\mathcal{Z}_0$
because, as it was shown in \cite{laa}, this class is contained in
$\mathcal{T}_d$.
\end{remark}

Many objects introduced below depend on $Z$ and $\ep$, although
sometimes we do not reflect this dependence in our notation.

Let $Z\in\mathcal{Z}_\ep$. We shall change the system of
coordinates in $\mathbb{R}^d$ twice. First we introduce in
$\mathbb{R}^d$ a new system of coordinates such that the unit
(Euclidean) ball $B_2^d$ of $\mathbb{R}^d$ is the maximal volume
ellipsoid in $Z$. From now on we consider the vectors $z_i$
introduced in Lemma \ref{L:shape} as vectors in $\mathbb{R}^d$ and
not as $d$-tuples of real numbers.
\medskip

It is easy to see that the support function of $Z$ is
given by
$$h_Z(x)=\sum_{i=1}^m|\langle x, z_i\rangle|.$$

It is more convenient for us to write this formula
in a different way. We consider the set
\begin{equation}\label{E:z_i}
\left\{\frac{z_1}{||z_1||},\dots,\frac{z_m}{||z_m||},-\frac{z_1}{||z_1||},
\dots,-\frac{z_m}{||z_m||}\right\}.
\end{equation}
If the vectors in (\ref{E:z_i}) are pairwise distinct, we let
$\mu$ to be the atomic measure on the unit (Euclidean) sphere $S$
whose atoms are given by
$\mu(z_i/||z_i||)=\mu(-z_i/||z_i||)=||z_i||/2$. It is easy to see
that
\begin{equation}\label{E:support}
h_Z(x)=\int_S|\langle x,z\rangle|d\mu(z).
\end{equation}

The defining formula for $\mu$ should be adjusted in the natural way
if some of the vectors in (\ref{E:z_i}) are equal.
\medskip

Conversely, if $\mu$ is a nonnegative measure on $S$ supported on
a finite set, then (\ref{E:support}) is a support function of some
zonotope (see \cite[Section 3.5]{Sch93} for more information on
this matter).
\medskip

Dealing with subsets of $S$ we use the following terminology and
notation. Let $x_0\in S$, $r>0$. The set $\Delta(x_0,r):=\{x\in
S:~||x-x_0||<r\hbox{ or }||x+x_0||<r\}$, where $||\cdot||$ is the
$\ell_2$-norm, is called a {\it cap}. If $0<r<\sqrt{2}$, then
$\Delta(x_0,r)$ consists of two connected components. In such a case
both $x_0$ and $-x_0$ will be considered as {\it centers} of
$\Delta(x_0,r)$.
\medskip

We are going to show that if $\ep>0$ is small, then the inequality
(\ref{E:min3}) implies that all but a very small part of the
measure $\mu$ is supported on a union of small caps centered at a
set of vectors which are multiples of a set of vectors satisfying
the condition: if we write their coordinates with respect to a
suitably chosen basis, we get a totally unimodular matrix. Having
such a set, it is easy to find $T\in\mt_d$ which is close to $Z$
with respect to the Banach--Mazur distance, see Lemma \ref{L:BM}.
\medskip

For any two numbers $\omega,\delta>0$ we introduce the set
$$\Omega(\omega,\delta):=\{x\in S:~ \mu(\Delta(x,\omega))\ge\delta\}$$
(recall that by $S$ we denote the unit sphere of $\ell_2^d$). In
what follows $c_1(d), c_2(d), \dots$, $C_1(d)$, $C_2(d), \dots$
denote quantities depending on the dimension $d$ only. Since $d$
is fixed throughout our argument, we regard them as constants.
\medskip

First we find conditions on $\omega$ and $\delta$ under which the
set $\Omega(\omega,\delta)$ contains a normalized basis
$\{e_i\}_{i=1}^d$ whose distance to an orthonormal basis can be
estimated in terms of $d$ only.

\begin{lemma}\label{L:bases} There exist $0<c_1(d), C_1(d), C_2(d)<\infty$, such that for
$\omega\le\frac1{6d}$ and $\delta\le c_1(d)\omega^{d-1}$ there is
a normalized basis $\{e_i\}_{i=1}^d$ in the space $\mathbb{R}^d$
satisfying the conditions:
\smallskip

\noindent{\bf (a)} $\mu(\Delta(e_i,\omega))\ge\delta$.
\smallskip

\noindent{\bf (b)} If $\{o_i\}_{i=1}^d$ is an orthonormal basis in
$\mathbb{R}^d$, then the operator $N:\mathbb{R}^d\to\mathbb{R}^d$
given by $No_i=e_i$ satisfies $||N||\le C_1(d)$ and $||N^{-1}||\le
C_2(d)$, where the norms are the operator norms of $N, N^{-1}$
considered as operators from $\ell_2^d$ into $\ell_2^d$.
\end{lemma}

\begin{proof}{Proof} We need an estimate for $\mu(S)$. Observe
that if $K_1$ and $K_2$ are two symmetric zonotopes and
$K_1\subset K_2$, then $\mu_1(S)\le\mu_2(S)$ for the corresponding
measures $\mu_1$ and $\mu_2$ (defined as even measures satisfying
(\ref{E:support}) with $Z=K_1$ and $Z=K_2$, respectively). To
prove this statement we integrate the equality (\ref{E:support})
with respect to $x$ over the Haar measure on $S$.
\medskip

Now we use the assumption that $B_2^d$ is the maximal volume
ellipsoid in $Z$. Let $\sum_{i=1}^n\gamma_ix_i\otimes x_i$ be the
F.~John representation of the identity operator corresponding to
$Z$ (see \cite[p.~46]{JL01}). Then
$$Z\subset\left\{x:~ |\langle x, x_i\rangle|\le 1~ \forall
i\in\{1,\dots,n\}\right\}.
$$
Since $\displaystyle{x=\sum_{i=1}^n\langle x, x_i\rangle
\gamma_ix_i}$ for each $x\in\mathbb{R}^d$, we have
$\displaystyle{Z\subset\sum_{i=1}^n[-\gamma_ix_i,\gamma_ix_i]}$.
Since $\displaystyle{\sum_{i=1}^n\gamma_i=d}$, this implies
$\mu(S)\le d$.
\medskip

Using the well-known computation, which goes back to B.~Gr\"unbaum
(\cite[p.~462, (5.2)]{Gru60}, see, also, \cite[pp.~94--95]{Jam87})
one can find estimates for $\mu(S)$ from below, which imply
$\mu(S)\ge\sqrt{d}$. For our purposes the trivial estimate
$\mu(S)\ge 1$ is sufficient (this estimate follows immediately
from $Z\supset B_2^d$, because this inclusion implies $h_Z(x)\ge
||x||$).
\medskip

We denote the normalized Haar measure on $S$ by $\eta$. It is well
known that there exists $c_2(d)>0$ such that
\begin{equation}\label{E:c_2}
\eta(\Delta(x,r))\ge c_2(d)r^{d-1}~\forall r\in(0,1)~\forall x\in
S.
\end{equation}
Using a standard averaging argument and $\mu(S)\ge 1$, we get that
there exists $e_1\in S$ such that
$$\mu(\Delta(e_1,\omega))\ge c_2(d)\omega^{d-1}.$$
\medskip

Consider the closed
$\displaystyle{\left(\frac1{3d}+\omega\right)}$-neighborhood (in
the $\ell_2^d$ metric) of the line $L_1$ spanned by $e_1$. Let
$\Delta_1$ be the intersection of this neighborhood with $S$. Our
purpose is to estimate $\mu(S\backslash \Delta_1)$ from below. Let
$x\in S$ be orthogonal to $e_1$. Then
$$1\le h_Z(x)\le 1\cdot\mu(S\backslash \Delta_1)+\left(\frac1{3d}+\omega\right)\cdot d,$$
where the left-hand side inequality follows from the fact that $Z$
contains $B_2^d$. Therefore $\mu(S\backslash \Delta_1)\ge
1-\displaystyle\left(\frac1{3d}+\omega\right)d$.
\medskip

We erase all measure $\mu$ contained in $\Delta_1$, use a standard
averaging argument again, and find a vector $e_2$ such that
$$\mu(\Delta(e_2,\omega)\backslash \Delta_1)\ge
c_2(d)\omega^{d-1}\left(1-\left(\frac1{3d}+\omega\right)d\right).$$
Since $\mu(\Delta(e_2,\omega)\backslash \Delta_1)>0$, the vector
$e_2$ is not in the $\frac1{3d}$-neighborhood of $L_1$.
\medskip

Let $\Delta_2$ be the intersection of $S$ with the closed
$\displaystyle{\left(\frac1{3d}+\omega\right)}$-neighborhood of
$L_2=\lin\{e_1,e_2\}$ (that is, $L_2$ is the linear span of
$\{e_1,e_2\}$). Let $x\in S$ be orthogonal to $L_2$. Then
$$1\le h_Z(x)\le 1\cdot\mu(S\backslash \Delta_2)+\left(\frac1{3d}+\omega\right)\cdot d,$$
where the left-hand side inequality follows from the fact that $Z$
contains $B_2^d$. Therefore $\mu(S\backslash \Delta_2)\ge
1-\displaystyle\left(\frac1{3d}+\omega\right)d$.
\medskip

Using the standard averaging argument in the same way as in the
previous step we find a vector $e_3$ such that
$$\mu(\Delta(e_3,\omega)\backslash \Delta_2)\ge
c_2(d)\omega^{d-1}\left(1-\left(\frac1{3d}+\omega\right)d\right).$$
Since $\mu(\Delta(e_3,\omega)\backslash \Delta_2)>0$, the vector
$e_3$ is not in the $\frac1{3d}$-neighborhood of $L_2$.
\medskip

We continue in an obvious way. As a result we construct a
normalized basis $\{e_1,\dots,e_d\}$ satisfying the conditions

\begin{itemize}
\item[{\bf (i)}] $\displaystyle{\mu(\Delta(e_i,\omega))\ge
c_2(d)\omega^{d-1}\left(1-\left(\frac1{3d}+\omega\right)d\right)}$.
\smallskip

\item[{\bf (ii)}]
$\displaystyle{\hbox{dist}(e_i,\hbox{lin}\{e_j\}_{j=1}^{i-1})\ge
\frac1{3d}}$, $i=2,\dots,d$, where $\hbox{dist}(\cdot,\cdot)$
denotes the distance from a vector to a subspace.
\end{itemize}

If $\omega<\frac1{6d}$, the inequality {\bf (i)} implies
$$\mu(\Delta(e_i,\omega))\ge\frac12
c_2(d)\omega^{d-1},$$ and we get the estimate {\bf (a)} of Lemma
\ref{L:bases} with $c_1(d)=c_2(d)/2$.

To estimate $||N||$ and $||N^{-1}||$, we let $\{o_i\}_{i=1}^d$ be
the basis obtained from $\{e_i\}$ using the Gram--Schmidt
orthonormalization process. Let $N:\mathbb{R}^d\to\mathbb{R}^d$ be
defined by $No_i=e_i$. The estimate $||N||\le C_1(d)$ with
$C_1(d)=\sqrt{d}$ follows because the vectors $\{e_i\}_{i=1}^d$
are normalized and the vectors $\{o_i\}_{i=1}^d$ form an
orthonormal set.
\medskip

To estimate $||N^{-1}||$ we observe that the matrix of $N$ with
respect to the basis $\{o_i\}$ is of the form
$$N=\left(\begin{array}{cccc} N_{11} & N_{12} &\dots & N_{1d}\\
0 & N_{22} & \dots & N_{2d}\\
\vdots &\vdots &\ddots &\vdots\\
0 & 0 & \dots & N_{dd}\end{array}\right),$$ and that the
inequality {\bf (ii)} implies $N_{ii}\ge \frac1{3d}$. We have
$$T=\left(\begin{array}{cccc} N_{11} & 0 &\dots & 0\\
0 & N_{22} & \dots & 0\\
\vdots &\vdots &\ddots &\vdots\\
0 & 0 & \dots & N_{dd}\end{array}\right)\cdot
\left(\begin{array}{cccc} 1 & \frac{N_{12}}{N_{11}} &\dots & \frac{N_{1d}}{N_{11}}\\
0 & 1 & \dots & \frac{N_{2d}}{N_{22}}\\
\vdots &\vdots &\ddots &\vdots\\
0 & 0 & \dots & 1\end{array}\right)=D(I+U),$$ where $I$ is the
identity matrix,
$$D=\left(\begin{array}{cccc} N_{11} & 0 &\dots & 0\\
0 & N_{22} & \dots & 0\\
\vdots &\vdots &\ddots &\vdots\\
0 & 0 & \dots & N_{dd}\end{array}\right),\hbox{ and } U=
\left(\begin{array}{ccccc} 0 & \frac{N_{12}}{N_{11}} &\dots & \frac{N_{1,{d-1}}}{N_{11}} & \frac{N_{1d}}{N_{11}}\\
0 & 0 & \dots & \frac{N_{2,{d-1}}}{N_{22}} & \frac{N_{2d}}{N_{22}}\\
\vdots &\vdots &\ddots &\vdots &\vdots\\
0 & 0 & \dots & 0 & \frac{N_{d-1,d}}{N_{d-1,d-1}}\\
 0 & 0 & \dots & 0 & 0
\end{array}\right)
$$
Therefore
\begin{equation}\label{E:inverse}
N^{-1}=(I+U)^{-1}D^{-1}=(I-U+U^2-\dots+(-1)^{d-1}U^{d-1})D^{-1},
\end{equation}
the identity $(I+U)^{-1}=(I-U+U^2-\dots+(-1)^{d-1}U^{d-1})$
follows from the obvious equality $U^d=0$. The definition of $U$
and $N_{ii}\ge \frac1{3d}$ imply that columns of $U$ are vectors
with Euclidean norm at most $3d$, hence  $||U||\le 3d^{\frac32}$.
Therefore the identity (\ref{E:inverse}) implies the following
estimate for $||N^{-1}||$:
$$||N^{-1}||\le \frac{||U||^d-1}{||U||-1}\cdot||D^{-1}||\le
\frac{3^dd^{\frac{3d}2}-1}{3d^{\frac32}-1}\cdot3d.
$$
Denoting the right-hand side of this inequality by $C_2(d)$ we get
the desired estimate.
\end{proof}

\begin{remark}{Remark} We do not need sharp estimates for $c_1(d),
C_1(d)$, and $C_2(d)$ because $d$ is fixed in our argument, and
the dependence on $d$ of the parameters involved in our estimates
is not essential for our proofs.
\end{remark}

We use the following notation: for a set $\Gamma\subset S$ and a
real number $r>0$ we denote the set $\{x\in S:~
\inf\{||x-y||:~y\in\Gamma\}\le r\}$ by $\Gamma_r$.

\begin{lemma}\label{L:complement} Let $c_2(d)$ be the constant from {\rm (\ref{E:c_2})}, then
$\mu(S\backslash ((\Omega(\omega,\delta))_\omega))\le
\displaystyle{\frac{\delta}{c_2(d)\omega^{d-1}}}$.
\end{lemma}

\begin{proof}{Proof} Assume the contrary, that is, $\mu(S\backslash ((\Omega(\omega,\delta))_\omega))>
\frac{\delta}{c_2(d)\omega^{d-1}}$. Then, using a standard
averaging argument as in Lemma \ref{L:bases}, we find a point $x$
such that
$$\mu(\Delta(x,\omega)\backslash ((\Omega(\omega,\delta))_\omega))\ge
c_2(d)\omega^{d-1}\cdot\frac{\delta}{c_2(d)\omega^{d-1}}=\delta.$$
By the definition of $\Omega(\omega,\delta)$ this implies $x\in
\Omega(\omega,\delta)$. On the other hand, since the set
$\Delta(x,\omega)\backslash ((\Omega(\omega,\delta))_\omega)$ is
non-empty, it follows that $x\notin \Omega(\omega,\delta)$. We get
a contradiction. \end{proof}

\subsection{Notation and definitions used in the rest of the
proof}\label{S:notation}

For each $Z\in\mathcal{Z}_\ep$ we apply Lemma \ref{L:bases} with
$\omega=\omega(\ep)=\ep^{4k}$ and $\delta=\delta(\ep)=\ep^{4dk}$,
where $0<k<1$ is a number satisfying the conditions
\begin{equation}
k<\frac{1}{6+4d^2}~\hbox{ and }~k<\frac{1}{2d+4d^2},
\end{equation}
we choose and fix such number $k$ for the rest of the proof. It is
clear that there is $\Xi_0=\Xi_0(d,k)>0$ such that the conditions
$\omega(\ep)\le\frac1{6d}$ and $\delta(\ep)\le
c_1(d)(\omega(\ep))^{d-1}$  are satisfied for all $\ep\in(0,\Xi_0)$,
where $c_1(d)$ is the constant from Lemma \ref{L:bases}. In the rest
of the argument we consider $\ep\in(0,\Xi_0)$ only. Let
$\{e_i\}_{i=1}^d$ be one of the bases satisfying the conditions of
Lemma \ref{L:bases} with the described choice of $\omega$ and
$\delta$. Now we change the system of coordinates in
$\mathbb{R}^d\supset Z$ the second time. The new system of
coordinates is such that $\{e_i\}_{i=1}^d$ is its unit vector basis.
We shall modify the objects introduced so far ($\Omega,\mu$, etc.)
and denote their versions corresponding to the new system of
coordinates by $\co, \check\mu$, etc. All these objects depend on
$Z$, $\ep$, and the choice of $\{e_i\}_{i=1}^d$.
\medskip

We denote by $\check S$ the Euclidean unit sphere in the new
system of coordinates. We denote by $\mn:S\to\check S$ the natural
normalization mapping, that is, $\mn(z)=z/||z||$, where $||z||$ is
the Euclidean norm of $z$ with respect to the new system of
coordinates. The estimates for $||N||$ and $||N^{-1}||$ from Lemma
\ref{L:bases} imply that the Lipschitz constants of the mapping
$\mn$ and its inverse $\mn^{-1}:\check S\to S$ can be estimated in
terms of $d$ only.

We introduce a measure $\cm$ on $\check S$ as an atomic measure
supported on a finite set and such that $\cm(\mn(z))=\mu(z)||z||$
for each $z\in S$ , where $||z||$ is the norm of $z$ in the new
system of coordinates. Using the definition of the zonotope $Z$ it
is easy to check that the function
$$\check h_Z(x)=\int_{\check S}|\langle x,\check
z\rangle|d\check\mu(\check z),$$ where $\langle\cdot,\cdot\rangle$
is the inner product in the new coordinate system, is the support
function of $Z$ in the new system of coordinates.
\medskip

We define $\co=\co(\omega,\delta)$ as
$\mn(\Omega(\omega,\delta))$. It is clear that $e_i\in\co$.
Everywhere below we mean coordinates in the new system of
coordinates (when we refer to $||\cdot||$, $\Delta$, etc).
\medskip

The observation that $\mn$ and $\mn^{-1}$ are Lipschitz, with
Lipschitz constants estimated in terms of $d$ only, implies the
following statements:

\begin{itemize}

\item There exist  $C_3(d), C_4(d)<\infty$ such that
\begin{equation}\label{E:Obs1} \cm(\check
S\backslash((\co(\omega,\delta))_{C_3(d)\omega(\ep)}))\le
C_4(d)\frac{\delta}{\omega^{d-1}}
\end{equation}
(we use Lemma \ref{L:complement}).

\item There exist $c_3(d)>0$ and $C_5(d)<\infty$ such that
\begin{equation}\label{E:Obs2}\cm(\Delta(x, C_5(d)\omega))\ge c_3(d)\delta~~\forall
x\in\co(\omega,\delta)\end{equation} (we use the definitions of
$\Omega(\omega,\delta)$ and $\co(\omega,\delta)$).

\item There exists a constant $C_6(d)$ depending on $d$ only, such
that
\begin{equation}\label{E:vol}
\vol(Z)\le C_6(d).\end{equation}

\end{itemize}

Let $\check Q$ be the transpose of the matrix whose columns are
the coordinates of $z_i$ in the new system of coordinates. We
denote by $\check w_i$ $(i=1,\dots,M)$ the $d\times d$ minors of
$\check Q$ ordered in the same way as the $w_i$. The vector
$\{\check w_i\}_{i=1}^M$ is a scalar multiple of
$\{w_i\}_{i=1}^M$. Therefore (\ref{E:equal}) implies
\begin{equation}\label{E:1+ep}
\max_i|u_i|\sum_{i=1}^M|\check w_i|=(1+\varepsilon)^d
\left|\sum_{i=1}^Mu_i\check w_i\right|.
\end{equation}
The volume of $Z$ in the new system of coordinates is
$2^d\sum_{i=1}^M|\check w_i|$.

\subsection{Lemma on six large minors}

To show that if $\ep>0$ is small, then the inequality (\ref{E:1+ep})
implies that all but a very small part of the measure $\cm$ is
supported ``around'' multiples of vectors represented by a totally
unimodular matrix in some basis, we need the following lemma. It
shows that the inequality (\ref{E:1+ep}) implies that the measure
$\cm$ cannot have non-trivial ``masses'' near $(d+2)$-tuples of
vectors satisfying certain condition.

\begin{lemma}\label{L:six} Let $\chi(\ep)$, $\sigma(\ep)$, and $\pi(\ep)$ be functions
satisfying the following conditions:

\begin{itemize}
\item[{\bf (1)}] $\displaystyle{\lim_{\ep\downarrow 0}\chi(\ep)=\lim_{\ep\downarrow
0}\sigma(\ep)=\lim_{\ep\downarrow 0}\pi(\ep)=0}$;

\item[{\bf (2)}] $\ep=o((\chi(\ep))^2(\sigma(\ep))^d)\hbox{
as }\ep\downarrow 0$;

\item[{\bf (3)}] $\pi(\ep)=o(\chi(\ep)) \hbox{ as }\ep\downarrow
0$;

\item[{\bf (4)}] There is a subset $\Phi_0\subset (0,\Xi_0)$ such
that the closure of $\Phi_0$ contains $0$, and for each
$\ep\in\Phi_0$ there exist $Z\in \mathcal{Z}_\ep$ and points
$x_1,\dots,x_{d-2},p_1,p_2,p_3,p_4$ in the corresponding $\check
S$, such that
\begin{equation}\label{E:measure}
\cm(\Delta(z,\pi(\ep)))\ge\sigma(\ep)~~\forall
z\in\{x_1,\dots,x_{d-2},p_1,p_2,p_3,p_4\}.
\end{equation}
\end{itemize}

Let $\mathcal{U}_0$ be the set of pairs $(\ep,Z)$ in which
$\ep\in\Phi_0$ and $Z$ satisfies the condition from {\bf (4)}.
\smallskip

Let $\Phi_1\subset\Phi_0$ be the set of those $\ep\in\Phi_0$ for
which there exists $(\ep,Z)\in\mathcal{U}_0$ such that the
corresponding points $x_1,\dots,x_{d-2},p_1,p_2,p_3,p_4$ satisfy
the condition
\begin{equation}\label{E:detge} |\det(H_{\alpha,\beta})|\ge\chi(\ep)\end{equation}
for all matrices $H_{\alpha,\beta}$ whose columns are the
coordinates of $\{x_1,\dots,x_{d-2},p_\alpha,p_\beta\}$,
$\alpha,\beta\in\{1,2,3,4\}$, $\alpha\ne\beta$, with respect to an
orthonormal basis $\{e_i\}_{i=1}^d$ in $\mathbb{R}^d$. Then there
exists $\Xi_1>0$ such that  $\Phi_1\cap(0,\Xi_1)=\emptyset$.
\end{lemma}

\begin{proof}{Proof} We assume the contrary, that is, we assume that
$0$ belongs to the closure of $\Phi_1$. For each $\ep\in\Phi_1$ we
choose $Z\in\mathcal{Z}_\ep$ such that $(\ep,Z)\in\mathcal{U}_0$
and the condition (\ref{E:measure}) is satisfied. We show that for
sufficiently small $\ep>0$ this leads to a contradiction.
\medskip

We consider the following perturbation of the matrix
$H_{\alpha,\beta}$: each column vector $z$ in it is replaced by a
vector from $\Delta(z,\pi(\ep))$. We denote the obtained
perturbation of the matrix $H_{\alpha,\beta}$ by
$H_{\alpha,\beta}^p$. We claim that
\begin{equation}\label{E:Hp}
|\det(H_{\alpha,\beta}^p)|\ge\chi(\ep)-d\cdot\pi(\ep).\end{equation}

To prove this claim we need the following lemma, which we state in
a bit more general form than is needed now, because we shall need
it later.

\begin{lemma}\label{L:detperturb} Let $x_1,\dots, x_d,z\in \ell_2^d$ be such that
$\displaystyle{\max_{2\le i\le d}}||x_i||\le\mathfrak{m}$ and
$||z-x_1||\le\mathfrak{l}$. Then
$$\left|\det[z,x_2,\dots,x_d]-\det[x_1,x_2,\dots,x_d]\right|\le\mathfrak{l}\cdot\mathfrak{m}^{d-1}.$$
\end{lemma}

This lemma follows immediately from the volumetric interpretation
of determinants.
\medskip

To get the inequality (\ref{E:Hp}) we apply Lemma
\ref{L:detperturb} $d$ times with $\mathfrak{m}=1$ and
$\mathfrak{l}=\pi(\ep)$.
\medskip

Since $Z\in\mathcal{Z}_\ep$, it can be represented in the form
$Z=\sum_iI(z_i)$. First we complete our proof in a special case
when the following condition is satisfied:
\smallskip

\noindent{\bf (*)} {\it All vectors $z_i$ whose normalizations
$z_i/||z_i||$ belong to the sets $\Delta(z, \pi(\ep))$, $z\in
\{x_1,\dots$, $x_{d-2},p_1,p_2,p_3, p_4\}$, have the same norm
$\tau$ and there are equal amounts of such vectors in each of the
sets $\Delta(z, \pi(\ep))$, $z\in \{x_1,\dots$,
$x_{d-2},p_1,p_2,p_3, p_4\}$, we denote the common value of the
amounts by $F$.}
\smallskip

The inequality (\ref{E:measure}) implies
$$F\cdot\tau\ge \sigma(\ep)$$

We denote by $\Lambda$ the set of all numbers $i\in\{1,\dots, M\}$
satisfying the condition: the normalizations of columns of the
minor $\check w_i$ form a matrix of the form $H_{\alpha,\beta}^p$,
for some $\alpha,\beta\in\{1,2,3,4\}$.
\medskip

We need an estimate for $\displaystyle{\sum_{i\in\Lambda}|\check
w_i|}$. The inequality (\ref{E:Hp}) implies $|\check
w_i|\ge\tau^d(\chi(\ep)-d\cdot\pi(\ep))$ for each $i\in\Lambda$.
\medskip

On the other hand, the cardinality $|\Lambda|$ of $\Lambda$ is
$6F^d$. In fact, there are $F^{d-2}$ ways to choose $z_i/||z_i||$
in the sets $\Delta(x_j,\pi(\ep))$, $j=1,\dots, d-2$. There are
$\binom{4}2=6$ ways to choose two of the sets
$\Delta(p_j,\pi(\ep))$, $j=1,2,3,4$, and there are $F^2$ ways to
choose one vector $z_i/||z_i||$ in each of them. Therefore
$|\Lambda|=6F^d$ and
\begin{equation}\label{E:2.16}
\sum_{i\in\Lambda}|\check w_i|\ge
6F^d\tau^d(\chi(\ep)-d\cdot\pi(\ep))\ge6(\sigma(\ep))^d(\chi(\ep)-d\cdot\pi(\ep)).
\end{equation}

We assume for simplicity that $\max_i|u_i|=1$ (if it is not the
case, some of the sums below should be multiplied by
$\max_i|u_i|$). The $u_i$ are defined above the equality
(\ref{E:equal}). Then the condition (\ref{E:1+ep}) can be
rewritten as
\begin{equation}\label{E:2.17}
(1+\varepsilon)^d\left|\sum_{i=1}^Mu_i\check w_i\right|\ge
\sum_{i\in\Lambda}|\check w_i|+\sum_{i\notin\Lambda}|\check w_i|.
\end{equation}

On the other hand,
\begin{equation}\label{E:2.18}
(1+\varepsilon)^d\left|\sum_{i=1}^Mu_i\check w_i\right|\le
(1+\varepsilon)^d\left|\sum_{i\in\Lambda}u_i\check w_i\right|+
(1+\varepsilon)^d\sum_{i\notin\Lambda}|\check w_i|.
\end{equation}

From (\ref{E:2.17}) and (\ref{E:2.18}) we get
\begin{equation}\label{E:2.19}
(1+\varepsilon)^d \left|\sum_{i\in\Lambda}u_i\check w_i\right|\ge
\sum_{i\in\Lambda}|\check w_i|- \left((1+\ep)^d-1\right)
\sum_{i\notin\Lambda}|\check w_i|.
\end{equation}

As is well known, $2^d\sum_{i=1}^M|\check w_i|$ is the volume of
$Z$, hence $\sum_{i=1}^M|\check w_i|\le 2^{-d}C_6(d)$.
\medskip

Using this observation and the inequalities (\ref{E:2.16}) and
(\ref{E:2.19}) we get
$$
\left|\sum_{i\in\Lambda}u_i\check w_i\right|\ge \left(
\frac1{(1+\varepsilon)^d}-
\frac{((1+\varepsilon)^d-1)C_6(d)2^{-d}}{6(\sigma(\ep))^d(\chi(\ep)-d\cdot\pi(\ep))}
\right) \sum_{i\in\Lambda}|\check w_i|.
$$
(We use the fact that $\chi(\ep)-d\cdot\pi(\ep)>0$ if $\ep>0$ is
small enough.) The conditions {\bf (2)} and {\bf (3)} imply that
there exists $\psi>0$ such that
\begin{equation}\label{E:0.04}
\left( \frac1{(1+\varepsilon)^d}-
\frac{((1+\varepsilon)^d-1)C_6(d)2^{-d}}{6(\sigma(\ep))^d(\chi(\ep)-d\cdot\pi(\ep))}
\right)>(1-0.04(\chi(\ep)-d\cdot\pi(\ep)))
\end{equation}
is satisfied if $\varepsilon\in(0,\psi)$. The right-hand side is
chosen in the form needed below.
\medskip

Let $\psi>0$ be such that the statement above is true. Then for
$\varepsilon\in(0,\psi)$ we have
\begin{equation}\label{E:not}
\left|\sum_{i\in\Lambda}u_i\check w_i\right|\ge
(1-0.04(\chi(\ep)-d\cdot\pi(\ep))) \sum_{i\in\Lambda}|\check w_i|.
\end{equation}

Recall that $u_i$ are $d\times d$ minors of some matrix
$[y_1,\dots,y_d]$. We need the Pl\"ucker relations, see
\cite[p.~312]{HP} or \cite[p.~42]{sh}. The result that we need can
be stated in the following way: if $\gamma_1,\dots,\gamma_{d-2},
\kappa_1, \kappa_2, \kappa_3, \kappa_4$ are indices of $d+2$ rows
of $[y_1,\dots,y_d]$, then
\begin{equation}\label{E:P}
t_{1,2}t_{3,4}-t_{1,4}t_{3,2}+t_{2,4}t_{3,1}=0,
\end{equation}
where $t_{\alpha,\beta}$ is the determinant of the $d\times d$
matrix whose rows are the rows of $[y_1,\dots,y_d]$ with the
indices $\gamma_1,\dots,\gamma_{d-2}, \kappa_\alpha$, and
$\kappa_\beta$. Note that (\ref{E:P}) can be verified by a
straightforward computation (which is very simple if we make a
suitable change of coordinates before the computation).
\medskip

Now we show that (\ref{E:not}) cannot be satisfied. Let $\Psi$ be
a set consisting of $(d+2)$ vectors $z_{\kappa_1}, z_{\kappa_2},
z_{\kappa_3}, z_{\kappa_4}, z_{\gamma_1}, \dots,
z_{\gamma_{d-2}}$, formed in the following way. We choose vectors
$(z_{\kappa_i}/||z_{\kappa_i}||)\in \Delta(p_i,\pi(\ep))$,
$i=1,2,3,4$, and choose vectors
$(z_{\gamma_i}/||z_{\gamma_i}||)\in \Delta(x_i,\pi(\ep))$,
$i=1,\dots,d-2$. To each such selection there corresponds a set of
$6$ minors $\check w_i$ of the form
$\tau^d\det(H_{\alpha,\beta}^p)$, we denote this set of six minors
by $\{\check w_i\}_{i\in M(\Psi)}$.
\medskip

One of the immediate consequences of the Pl\"ucker relation
(\ref{E:P}) is that for any such $(d+2)$-tuple $\Psi$
\begin{equation}\label{E:sqrt2}
|u_i|\le\frac1{\sqrt{2}}~\hbox{ for some }~i\in M(\Psi).
\end{equation}
(Here we use the assumption that $\max_i|u_i|=1$.)
\medskip

For each $\Psi$ we choose one such $i\in M(\Psi)$ and denote it by
$s(\Psi)$. The estimate (\ref{E:Hp}) and the condition {\bf (*)}
imply that
\begin{equation}\label{E:rho}
\tau^d\ge |\check w_i|\ge\tau^d(\chi(\ep)-d\cdot\pi(\ep))
\end{equation}
for every $i\in\Lambda$.
\medskip

Hence for every $(d+2)$-tuple $\Psi$ of the described type we have
\begin{align*}\begin{split}\left|\sum_{i\in M(\Psi)}u_i\check w_i\right|&\le \sum_{i\in
M(\Psi)\backslash\{s(\Psi)\}}|\check w_i|+ \frac1{\sqrt{2}}|\check
w_{s(\Psi)}|\le \sum_{i\in M(\Psi)}|\check
w_i|-\frac{\sqrt{2}-1}{\sqrt{2}} |\check w_{s(\Psi)}|\\
&=\sum_{i\in M(\Psi)}|\check w_i|\left(1-
\frac{(\sqrt{2}-1)|\check
w_{s(\Psi)}|}{\sqrt{2} \sum_{i\in M(\Psi)}|\check w_i|}\right)\\
&\le \sum_{i\in M(\Psi)}|\check w_i| \left(1-
\frac{(\sqrt{2}-1)\tau^d(\chi(\ep)-d\cdot\pi(\ep))}{\sqrt{2}
\cdot6\tau^d}\right)\\
&< \sum_{i\in M(\Psi)}|\check
w_i|\left(1-0.04(\chi(\ep)-d\cdot\pi(\ep))\right).\end{split}\end{align*}
Thus
\begin{equation}\label{E:quad}
\left|\sum_{i\in M(\Psi)}u_i\check w_i\right| < \sum_{i\in
M(\Psi)}|\check w_i|\left(1-0.04(\chi(\ep)-d\cdot\pi(\ep))\right).
\end{equation}

Recall that $F$ is the number of vectors $z_i$ corresponding to
each of the sets $\Delta(z, \pi(\ep))$, $z\in
\{x_1,\dots,x_{d-2},p_1,p_2,p_3, p_4\}$. Simple counting shows
that for an arbitrary collection $\{\Upsilon_i\}_{i\in\Lambda}$ of
numbers we have
$$
\sum_\Psi\sum_{i\in M(\Psi)}\Upsilon_i=F^2
\sum_{i\in\Lambda}\Upsilon_i.
$$

Using (\ref{E:quad}) we get that
\begin{align*}\begin{split}
F^2\left|\sum_{i\in\Lambda}u_i\check w_i\right|&=
\left|\sum_\Psi\sum_{i\in M(\Psi)}u_i\check w_i\right|\le
\sum_\Psi\left|\sum_{i\in M(\Psi)}u_i\check w_i\right|\\
&<\sum_\Psi\sum_{i\in M(\Psi)}|\check w_i|
(1-0.04(\chi(\ep)-d\cdot\pi(\ep)))\\
&= F^2\sum_{i\in\Lambda}|\check w_i|
(1-0.04(\chi(\ep)-d\cdot\pi(\ep))).
\end{split}\end{align*}

If $\varepsilon\in(0,\psi)$, we get a contradiction with
(\ref{E:not}).
\medskip

To see that the general case can be reduced to the case {\bf (*)}
we need the following observation:
\medskip

Let $\tau_1, \tau_2>0$ be such that $\tau_1+\tau_2=1$. We replace
the row with the coordinates of $z_j$ in $\check Q$ by two rows,
one of them is the row of coordinates of $\tau_1 z_j$ and the
other is the row of coordinates of $\tau_2 z_j$. The zonotope
generated by the rows of the obtained matrix coincides with $Z$.
In the matrix $[y_1,\dots, y_d]$ we replace the $j^{\rm th}$ row
by two copies of it. It is easy to see that if we replace the
sequences $\{u_i\}_{i=1}^M$ and $\{\check w_i\}_{i=1}^M$ by
sequences of $d\times d$ minors of these new matrices, the
condition (\ref{E:1+ep}) is still satisfied.
\medskip

We can repeat this `cutting' of vectors $z_j$ into `pieces' with
(\ref{E:1+ep}) still being valid.
\medskip

Therefore, we may assume the following: among $z_j$ corresponding
to each of the sets $\Delta(z, \pi(\ep))$, $z\in
\{x_1,\dots,x_{d-2},p_1,p_2,p_3, p_4\}$ there exists a subset
$\Phi(z, \pi(\ep))$ consisting of vectors having the same length
$\tau$, and such that the sum of norms of vectors from $\Phi(z,
\pi(\ep))$ is $\ge\displaystyle{\frac{\sigma(\varepsilon)}2}$,
moreover, we may assume that the numbers of such vectors in the
subsets $\Phi(z,\pi(\ep))$ are the same for all $z\in
\{x_1,\dots,x_{d-2},p_1,p_2,p_3, p_4\}$.
\medskip

Lemma \ref{L:six} in this case can be proved using the same
argument as before, but with $\Lambda$ being the set of those
minors $\check w_i$ for which rows are from $\Phi(z,\pi(\ep))$.
Everything starting with the inequality (\ref{E:2.16}) can be
shown in the same way as before; only some constants will be
changed (because we need to replace $\sigma(\varepsilon)$ by
$\frac{\sigma(\varepsilon)}2$).
\end{proof}

\subsection{Searching for a totally unimodular matrix}

Let $\rho(\ep)=\ep^k$, $\nu(\ep)=\ep^{3k}$. For a vector $s$ we
denote its coordinates with respect to $\{e_i\}_{i=1}^d$ by
$\{s_i\}_{i=1}^d$. (Here $k$ and $\{e_i\}_{i=1}^d$ are the same as
in Section \ref{S:notation}.)

\begin{lemma}\label{L:smallangles} If
\begin{equation}\label{E:k1}
k<\frac{1}{6+4d^2},
\end{equation}
then there exists $\Xi_2>0$ such that for $\ep\in(0,\Xi_2)$,
$s,t\in\co(\omega(\ep), \delta(\ep))$, and
$\alpha,\beta\in\{1,\dots,d\}$, the inequality
\begin{equation}\label{E:aboverho}
\min\{|s_\alpha|, |s_\beta|, |t_\alpha|, |t_\beta|\}\ge\rho(\ep),
\end{equation}
implies
\begin{equation}\label{E:largedet}
\left|\det\left(\begin{array}{cc} s_\alpha &
t_\alpha\\
s_\beta & t_\beta\end{array}\right)\right|<\nu(\ep).
\end{equation}
\end{lemma}

\begin{proof}{Proof} Assume the contrary, that is, there exists a subset $\Phi_2\subset (0,1)$, having $0$ in its
closure and such that for each $\ep\in\Phi_2$ there exist
$Z\in{\cal Z}_\ep$, $s,t\in\co(\omega(\ep), \delta(\ep))$ and
$\alpha,\beta$ satisfying the condition (\ref{E:aboverho}), and
such that
$$\left|\det\left(\begin{array}{cc} s_\alpha &
t_\alpha\\
s_\beta & t_\beta\end{array}\right)\right|\ge\nu(\ep).$$  We apply
Lemma \ref{L:six} with $\{x_1,\dots,x_{d-2}\}=
\{e_i\}_{i\ne\alpha,\beta}$,
$\{p_1,p_2,p_3,p_4\}=\{e_\alpha,e_\beta, s, t\}$. Using a
straightforward determinant computation we see that the condition
(\ref{E:detge}) is satisfied with $\chi(\ep)=\min\{1,\rho(\ep),
\nu(\ep)\}=\ep^{3k}$ (we consider $\ep<1$).
\medskip

The inequality (\ref{E:Obs2}) implies that the condition {\bf (4)}
of Lemma \ref{L:six} is satisfied with
$\pi(\ep)=C_5(d)\omega(\ep)=C_5(d)\ep^{4k}$ and
$\sigma(\ep)=c_3(d)\delta(\ep)=c_3(d)\ep^{4dk}$. It is clear that
the conditions {\bf (2)} and {\bf (3)} of Lemma \ref{L:six} are
satisfied. To get {\bf (2)} we use the condition (\ref{E:k1}).
Applying Lemma \ref{L:six}, we get the existence of the desired
$\Xi_2$.
\end{proof}

For each vector from $\co(\omega(\ep), \delta(\ep))$ we define its
{\it top set} as the set of indices of coordinates whose absolute
values $\ge\rho(\ep)$.
\medskip

The collection of all possible top sets is a subset of the set of
all subsets of $\{1,\dots,d\}$, hence its cardinality is at most
$2^d$. We create a collection $\Theta(\omega(\ep),
\delta(\ep))\subset\co(\omega(\ep), \delta(\ep))$ in the following
way: for each subset of $\{1,\dots,d\}$ which is a top set for at
least one vector from $\co(\omega(\ep), \delta(\ep))$, we choose
one of such vectors; the set $\Theta(\omega(\ep), \delta(\ep))$ is
the set of all vectors selected in this way.
\medskip

In our next lemma we show that each vector from $\co(\omega(\ep),
\delta(\ep))$ can be reasonably well approximated by a vector from
$\Theta(\omega(\ep), \delta(\ep))$. Therefore (as we shall see
later), to prove Lemma \ref{L:APPROX} it is sufficient to find a
``totally unimodular'' set approximating $\Theta(\omega(\ep),
\delta(\ep))$.

\begin{lemma}\label{L:smalldistance} Let $\rho(\ep)$ and $\nu(\ep)$ be as above and let $k$ and $\Xi_2$ be numbers
satisfying the conditions of Lemma {\rm \ref{L:smallangles}}. Let
$\ep\in(0,\Xi_2)$, $Z\in{\cal Z}_\ep$, and let
$s,t\in\co(\omega(\ep),\delta(\ep))$ be two vectors with the same
top set $\Sigma$. Then
\begin{equation}\label{E:topset}
\min\{||t+s||,
||t-s||\}\le\sqrt{2\frac{\nu(\ep)}{(\rho(\ep))^2}+4d\rho(\ep)^2}.
\end{equation}
\end{lemma}

\begin{proof}{Proof} Observe that if
$\rho(\ep)=\ep^k>\displaystyle{\frac{1}{\sqrt{d}}}$, the statement
of the lemma is trivial. Therefore we may assume that
$\rho(\ep)\le\displaystyle{\frac{1}{\sqrt{d}}}$. In such a case
$\Sigma$ contains at least one element.

First we show that the signs of different components of $s$ and
$t$ ``agree'' on $\Sigma$ in the sense that either they are the
same everywhere on $\Sigma$, or they are the opposite everywhere
on $\Sigma$. In fact, assume the contrary, and let $\alpha,
\beta\in\Sigma$ be indices for which the signs ``disagree''. Then,
as is easy to check,
$$\left|\det\left(\begin{array}{cc} s_\alpha &
t_\alpha\\
s_\beta &
t_\beta\end{array}\right)\right|=|s_\alpha||t_\beta|+|s_\beta||t_\alpha|\ge
2(\rho(\ep))^2>\nu(\ep),$$ and we get a contradiction. We consider
the case when the signs of $t_\alpha$ and $s_\alpha$ are the same
for each $\alpha\in\Sigma$, the other case can be treated
similarly (we can just consider $-s$ instead of $s$).
\medskip

We may assume without loss of generality that $|t_\alpha|\ge
|s_\alpha|$ for some $\alpha\in\Sigma$. We show that in this case
\begin{align*}
|t_\beta|\ge\left(1-\frac{\nu(\ep)}{(\rho(\ep))^2}\right)|s_\beta|
\end{align*}
for all $\beta\in\Sigma$. In fact, if
$|t_\beta|<\displaystyle{\left(1-\frac{\nu(\ep)}{(\rho(\ep))^2}\right)}|s_\beta|$
for some $\beta\in\Sigma$, then
$$\nu(\ep)>\left|\det\left(\begin{array}{cc} s_\alpha &
t_\alpha\\
s_\beta & t_\beta\end{array}\right)\right|\ge|t_\alpha|
|s_\beta|-|s_\alpha||t_\beta|\ge|s_\alpha||s_\beta|\frac{\nu(\ep)}{(\rho(\ep))^2}\ge\nu(\ep),
$$
a contradiction.
\medskip

We have
\begin{align*}\begin{split}
||t-s||^2&=||t||^2+||s||^2-2\langle t, s\rangle\le
2-2\sum_{\alpha\in\Sigma}\left(1-\frac{\nu(\ep)}{(\rho(\ep))^2}\right)s_\alpha^2+2\sum_{\alpha\notin\Sigma}\rho(\ep)^2\\
&\le
2\frac{\nu(\ep)}{(\rho(\ep))^2}+4\sum_{\alpha\notin\Sigma}\rho(\ep)^2\le
2\frac{\nu(\ep)}{(\rho(\ep))^2}+4d\rho(\ep)^2.\end{split}
\end{align*} Q.E.D.\end{proof}

Let $\Theta(\omega(\ep),\delta(\ep))=\{\mathfrak{b}_j\}_{j=1}^J$,
where $J\le 2^d$. We may and shall assume that
$\{e_i(\ep)\}_{i=1}^d\subset \Theta(\omega(\ep),\delta(\ep))$ (see
Lemma \ref{L:bases} and Section \ref{S:notation}). We denote
$d\cdot 2^d$ by $\mathfrak{n}$ and introduce $d\cdot\mathfrak{n}$
functions:
$\varphi_1(\ep),\dots,\varphi_{d\cdot\mathfrak{n}}(\ep)$, such
that
\begin{equation}\label{E:phi1}
\varphi_1(\ep)\ge\dots\ge\varphi_{d\cdot
\mathfrak{n}}(\ep)=\rho(\ep)=\ep^k.
\end{equation}

\begin{equation}\label{E:phi2}
\displaystyle{\varphi_{\alpha}(\ep)=(\varphi_{\alpha+1}(\ep))^{\frac1{d+1}}}.
\end{equation}

We consider the matrix $X$ whose columns are
$\{\mathfrak{b}_j\}_{j=1}^J$. We order the absolute values of
entries of this matrix in non-increasing order and denote them by
$\ma_1\ge\ma_2\ge\dots\ge\ma_{d\cdot J}$. Let $j_0$ be the least
index for which
\begin{equation}\label{E:j_0}
\varphi_{d\cdot j_0}(\ep)>\ma_{j_0}. \end{equation} The existence
of $j_0$ follows from $\{e_i(\ep)\}_{i=1}^d\subset
\Theta(\omega(\ep),\delta(\ep))$. The definition of $j_0$ implies
that $\ma_j\ge \varphi_{d\cdot j}(\ep)$ for $j<j_0$, hence
$\ma_j\ge\varphi_{d\cdot(j_0-1)}(\ep)$ for $j\le j_0-1$.
\medskip

We replace all entries of the matrix $X$ except
$\ma_1,\dots,\ma_{j_0-1}$ by zeros and denote the obtained matrix
by $G=(G_{ij})$, $i=1,\dots,d$, $j=1,\dots,J$, and its columns by
$\{g_j\}_{j=1}^{J}$. It is clear that
\begin{equation}\label{E:gx}||g_j-\mathfrak{b}_j||\le d\cdot\varphi_{dj_0}(\ep).
\end{equation}

We form a bipartite graph $\mathcal{G}$ on the vertex set $\{\bar
1, \dots,\bar d\}\cup\{1,\dots,J\}$, where we use bars in $\bar 1,
\dots, \bar d$ because these vertices are considered as different
from the vertices $1,\dots,d$, which are in the set
$\{1,\dots,J\}$. The edges of $\mathcal{G}$ are defined in the
following way: the vertices $\bar i$ and $j$ are adjacent if and
only if $G_{ij}\ne 0$. So there is a one-to-one correspondence
between edges of $\mathcal{G}$ and non-zero entries of $G$. We
choose and fix a maximal forest $\mathcal{F}$ in $\mathcal{G}$.
(We use the standard terminology, see, e. g. \cite[p.~11]{Sch86}.)
\medskip

For each non-zero entry of $G$ we define its {\it level} in the
following way:
\smallskip

The level of entries corresponding to edges of $\mathcal{F}$ is
$1$.
\medskip

For a non-zero entry of $G$ which does not correspond to an edge
in $\mathcal{F}$ we consider the cycle in $\mathcal{G}$ formed by
the corresponding edge and edges of $\mathcal{F}$. We define the
{\it level} of the entry as the half of the length of the cycle
(recall that the graph $\mathcal{G}$ is bipartite, hence all
cycles are even).
\medskip

\noindent{\bf Observation.} One of the classes of the bipartition
has $d$ vertices. Hence no cycle can have more than $2d$ edges,
and the level of each vertex is at most $d$.
\medskip

To each entry $G_{ij}$ of level $f$ we assign a square submatrix
$G(ij)$ of $G$ all other entries in which are of levels at most
$f-1$. We do this in the following way. To entries corresponding
to edges of $\mathcal{F}$ we assign the $1\times 1$ matrices
containing these entries. For an entry $G_{ij}$ which does not
correspond to an edge in $\mathcal{F}$ we consider the
corresponding edge $\mathfrak{e}$ in $\mathcal{G}$ and the cycle
$\mathcal{C}$ formed by $\mathfrak{e}$ and edges of $\mathcal{F}$.
Then we consider the entries in $G$ corresponding to edges of
$\mathcal{C}$ and the minimal submatrix in $G$ containing all of
these entries. Now we consider all edges in $\mathcal{G}$
corresponding to non-zero entries of this submatrix. We choose and
fix in this set of edges a minimum-length cycle $\mathcal{M}$
containing $\mathfrak{e}$. We define $G(ij)$ as the minimal
submatrix of $G$ containing all entries corresponding to edges of
$\mathcal{M}$. It is easy to verify that:
\begin{itemize}
\item $G(ij)$ is a square submatrix of $G$. \item Non-zero entries
of $G(ij)$ are in one-to-one correspondence with entries of
$\mathcal{M}$. \item The expansion of the determinant of $G(ij)$
according to the definition contains exactly two non-zero terms.
\item All non-zero entries of $G(ij)$ except $G_{ij}$ have level
$\le f-1$.
\end{itemize}

\begin{lemma}\label{L:det0} Let $k<1/(2d+4d^2)$.
If $\ep>0$ is small enough, then there exists a $d\times J$ matrix $\tilde G$ such that:
\smallskip

\noindent{\rm (1)} If some entry of $G$ is zero, the corresponding
entry of $\tilde G$ is also zero.
\smallskip

\noindent{\rm (2)} The entries of level $1$ of $\tilde G$ are the
same as for $G$;
\smallskip

\noindent{\rm (3)} All other non-zero entries of $\tilde G$ are
perturbations of entries of $G$ satisfying the following
conditions:

\begin{itemize}

\item[{\rm (a)}] If $G_{ij}$ is of level $f$, then $|G_{ij}-\tilde
G_{ij}|<\varphi_{d\cdot j_0-f+1}(\ep)$.

\item[{\rm (b)}] For each non-zero entry $G_{ij}$ of level $\ge 2$
of $G$ the determinant of the submatrix  $\tilde G(ij)$ of $\tilde
G$ corresponding to $G(ij)$ is zero.
\end{itemize}
\end{lemma}

\begin{proof}{Proof} Let $G_{ij}$ be an entry of level $f$. Since,
as it was observed above, all entries of $G(ij)$ have level $\le
f-1$, we can prove the lemma by induction as follows.

(1) We let $\tilde G_{ij}=G_{ij}$ for all $G_{ij}$ of level one.

(2) Let $f\ge 2$. {\bf Induction hypothesis:} We assume that for
all entries $G_{ij}$ of levels $\ell(G_{ij})$ satisfying
$2\le\ell(G_{ij})\le f-1$ we have found perturbations $\tilde
G_{ij}$ satisfying
$$|G_{ij}-\tilde G_{ij}|\le \varphi_{d\cdot j_0-\ell(G_{ij})+1}(\ep),$$
such that $\det(\tilde G(ij))=0$. (Note that this assumption is
vacuous if $f=2$.)
\medskip

{\bf Inductive step:} Let $G_{ij}$ be an entry of level $f$. If
$\ep>0$ is small enough we can find a number $\tilde G_{ij}$ such
that $|\tilde G_{ij}-G_{ij}|\le\varphi_{d\cdot j_0-f+1}(\ep)$ and
$\det(\tilde G(ij))=0$. Observe that by the induction hypothesis
and the observation that all other entries of $G(ij)$ have levels
$\le f-1$, all other entries of $\tilde G(ij)$ have already been
defined.
\medskip

So let $G_{ij}$ be an entry of level $f$, and $G(ij)$ be the
corresponding square submatrix. Renumbering rows and columns of
the matrix $G$ we may assume that the matrix $G(ij)$ looks like
the one sketched below for some $h\le f$.

$$G(ij)=\left(\begin{array}{ccccc}
a_1 & 0 & \dots & 0 & G_{ij}\\
b_1 & a_2 & \dots & 0 & 0 \\
\vdots & \vdots & \ddots & \vdots & \vdots\\
0 & 0 & \dots & a_{h-1} & 0\\
0 & 0 & \dots & b_{h-1} & a_h
\end{array}\right)$$

Therefore the matrix $G$ (possibly, after renumbering of columns
and rows) has the form
\begin{equation}\label{E:G}
\left(\begin{array}{ccccccccccccc} a_1 & 0 & \dots & 0 & G_{ij} &
0 & 0 & \dots & 0 & 0 & 1 & 0 &
\dots\\
b_1 & a_2 & \dots & 0 & 0 & 0 & 0 & \dots & 0 & 0 & 0 & 1 &
\dots\\
\vdots & \vdots & \ddots & \vdots & \vdots & \vdots & \vdots &
\ddots & \vdots & \vdots &\vdots & \vdots & \ddots \\
0 & 0 & \dots & a_{h-1} & 0 & 0 & 0 & \dots & 0 & 0 & 0 & 0 &
\dots\\
0 & 0 & \dots & b_{h-1} & a_h & 0 & 0 & \dots & 0 & 0 & 0 & 0 &
\dots\\
* & * & \dots & * & * & 1 & 0 & \dots & 0 & 0 & 0 & 0 &
\dots\\
* & * & \dots & * & * & 0 & 1 & \dots & 0 & 0 & 0 & 0 &
\dots\\
\vdots & \vdots & \ddots & \vdots & \vdots & \vdots & \vdots &
\ddots & \vdots & \vdots &\vdots & \vdots & \ddots \\
* & * & \dots & * & * & 0 & 0 & \dots & 1 & 0 & 0 & 0 &
\dots\\
* & * & \dots & * & * & 0 & 0 & \dots & 0 & 1 & 0 & 0 &
\dots
\end{array}\right)\end{equation}

We have assumed that we have already found entries $\{\tilde
a_n\}_{n=1}^h$ and $\{\tilde b_n\}_{n=1}^{h-1}$ of $\tilde G$
which are perturbations of $\{a_n\}_{n=1}^h$ and $\{
b_n\}_{n=1}^{h-1}$. The entries $1$ shown (\ref{E:G}) are the only
non-zero entries in their columns, therefore the corresponding
edges of $\mathcal{G}$ should be in $\mathcal{F}$. Let us denote
the perturbation of $G_{ij}$ we are looking for by $\tilde
G_{ij}$. The condition (b) of Lemma \ref{L:det0} can be written as
\begin{equation}\label{E:tilde}\prod_{n=1}^h\tilde
a_n+(-1)^{h-1}\prod_{n=1}^{h-1}\tilde b_n\cdot \tilde
G_{ij}=0\end{equation}

So it suffices to show that the number $\tilde G_{ij}$, found as a
solution of (\ref{E:tilde}) satisfies $|\tilde
G_{ij}-G_{ij}|<\varphi_{d\cdot j_0-f+1}(\ep)$. To show this we
assume the contrary. Since there are finitely many possibilities
for $j_0$ and $f$, the converse can be described as existence of
$j_0$ and $f$, such that there is a subset $\Phi_3\subset(0,1)$,
whose closure contains $0$, satisfying the condition:

For each $\ep\in\Phi_3$ there is $Z\in\mathcal{Z}_\ep$ such that
after proceeding with all steps of the construction we get: all
the conditions above are satisfied, but
\begin{equation}\label{E:dettilde}
\left|\prod_{n=1}^h\tilde a_n+(-1)^{h-1}\prod_{n=1}^{h-1}\tilde
b_n\cdot G_{ij}\right|>\varphi_{d\cdot
j_0-f+1}(\ep)\prod_{n=1}^{h-1}|\tilde b_n|.
\end{equation}

We need to get from here an estimate for $|\det(G(ij))|$ from
below. To get it we observe that the inequality (\ref{E:dettilde})
is an estimate from below of the determinant of the matrix

$$G'(ij)=\left(\begin{array}{ccccc}
\tilde a_1 & 0 & \dots & 0 & G_{ij}\\
\tilde b_1 & \tilde a_2 & \dots & 0 & 0 \\
\vdots & \vdots & \ddots & \vdots & \vdots\\
0 & 0 & \dots & \tilde a_{h-1} & 0\\
0 & 0 & \dots & \tilde b_{h-1} & \tilde a_h
\end{array}\right).$$

To get from here an estimate for $\det(G(ij))$ from below we
observe the following: The $\ell_2$-norm of each column of
$G_{ij}$ is $\le 1$, the $\ell_2$-distance between a column of
$G_{ij}$ and the corresponding column of $G'(ij)$ is at most
$2\varphi_{dj_0-f+2}(\ep)$. Hence the $\ell_2$-norm of each column
of $G'(ij)$ is $\le 1+2\varphi_{dj_0-f+2}(\ep)$. Applying Lemma
\ref{L:detperturb} $h$ times we get
$$|\det(G(ij))|\ge|\det(G'(ij))|-h\cdot
2\varphi_{dj_0-f+2}(\ep)(1+2\varphi_{dj_0-f+2}(\ep))^{h-1}.
$$

The induction hypothesis implies
$$|\tilde b_i|\ge\varphi_{d(j_0-1)}(\ep)-\varphi_{dj_0-f+2}(\ep),
$$
we get
\begin{equation}\label{E:detG}\begin{split}
|\det(G(ij))|&\ge
\varphi_{dj_0-f+1}(\ep)\cdot(\varphi_{d(j_0-1)}(\ep)-\varphi_{dj_0-f+2}(\ep))^{h-1}\\
&\quad-h\cdot
2\varphi_{dj_0-f+2}(\ep)(1+2\varphi_{dj_0-f+2}(\ep))^{h-1}.\end{split}
\end{equation}

Let us keep the notation $\{g_j\}_{j=1}^J$ for columns of the
matrix (\ref{E:G}). We consider the following six $d\times d$
minors of this matrix: the corresponding submatrices contain the
columns $\{g_2,\dots,g_{h-1},g_{h+1},\dots,g_d\}$, and two out of
the four columns $\{g_1,g_h,g_{d+1}, g_{d+2}\}$. Observe that
$g_{h+1}=e_{h+1},\dots,g_d=e_d,g_{d+1}=e_1, g_{d+2}=e_2$.
\medskip

The absolute values of the minors are equal to
\begin{equation}\label{E:minors}
|\det G(ij)|, ~\left|\prod_{n=2}^h a_n\right|,
~\left|\prod_{n=1}^{h-1} b_n\right|,
~|a_1|\cdot\left|\prod_{n=2}^{h-1} b_n\right|,
~\left|\prod_{n=2}^h b_n\right|, ~\left|\prod_{n=2}^{h-1}
b_n\right|.
\end{equation}

The first number in (\ref{E:minors}) was estimated in
(\ref{E:detG}). All other numbers are at least
$(\varphi_{d(j_0-1)}(\ep))^{h-1}$, it is clear that this number
exceeds the number from (\ref{E:detG}).
\medskip

We are going to use Lemma \ref{L:six} with $\{x_1,\dots,x_{d-2}\}=
\{\mathfrak{N}(g_2),\dots,\mathfrak{N}(g_{h-1}),\mathfrak{N}(g_{h+1}),\dots$,
$\mathfrak{N}(g_d)\}$ and
$\{p_1,p_2,p_3,p_4\}=\{\mathfrak{N}(g_1),\mathfrak{N}(g_h),\mathfrak{N}(g_{d+1}),
\mathfrak{N}(g_{d+2})\}$. (Recall that $\mathfrak{N}(z)=z/||z||$.)
Our definitions imply that $||\mathfrak{b}_j||=1$ and $||g_j||\le
1$, because $g_j$ is obtained from $\mathfrak{b}_j$ by replacing
some of the coordinates by zeros. Hence the inequality
(\ref{E:detG}) and the remark above on the numbers
(\ref{E:minors}) imply that the condition (\ref{E:detge}) is
satisfied with
\begin{equation}\label{E:chi}\begin{split} \chi(\ep)&=
\varphi_{dj_0-f+1}(\ep)\cdot(\varphi_{d(j_0-1)}(\ep)-\varphi_{dj_0-f+2}(\ep))^{h-1}\\
&\quad-h\cdot
2\varphi_{dj_0-f+2}(\ep)(1+2\varphi_{dj_0-f+2}(\ep))^{h-1}.
\end{split}
\end{equation}

The inequality (\ref{E:gx}), the inclusion
$\mathfrak{b}_j\in\co(\omega(\ep),\delta(\ep))$ and (\ref{E:Obs2})
imply that the condition (\ref{E:measure}) is satisfied with
$\pi(\ep)=2d\cdot\varphi_{dj_0}(\ep)+C_5(d)\omega(\ep)$ and
$\sigma(\ep)=c_3(d)\delta(\ep)$. So it remains to show that the
condition (\ref{E:phi2}) implies that the conditions {\bf (2)} and
{\bf (3)} of Lemma \ref{L:six} are satisfied.

By (\ref{E:phi2}), (\ref{E:chi}), the inequality $2\le h\le f\le
d$, and the trivial observation that all functions
$\varphi_\alpha(\ep)$ do not exceed $1$ for $0\le \ep\le 1$, we
have
\begin{equation}\label{E:Ochi}(\varphi_{dj_0-f+1}(\ep))^d=O(\chi(\ep)).
\end{equation}

Now we verify the condition {\bf (3)} of Lemma \ref{L:six}. The
part (b) can be verified as follows. The conditions (\ref{E:phi1})
and (\ref{E:phi2}), together with $f\ge 2$ and
$\omega(\ep)=\ep^{4k}$, imply that
$\pi(\ep)=O(\varphi_{dj_0}(\ep))=o((\varphi_{dj_0-f+1}(\ep))^d)=o(\chi(\ep))$.

To verify the condition {\bf (2)} of Lemma \ref{L:six} it suffices
to observe that (\ref{E:Ochi}) and (\ref{E:phi1}) imply
$(\rho(\ep))^d=O(\chi(\ep))$. Hence {\bf (2)} is satisfied if
$2dk+4d^2k<1$. This inequality is among the conditions of Lemma
\ref{L:det0}. Hence we can apply Lemma \ref{L:six} and get the
conclusion of Lemma \ref{L:det0}.
\end{proof}

Now let $\tilde G$ be an approximation of $G$ by a matrix
satisfying the conditions of Lemma \ref{L:det0}. We use the same
maximal forest $\mathcal{F}$ in $\mathcal{G}$ as above.  It is
easy to show (and the corresponding result is well known in the
theory of matroids, see, for example, \cite[Theorem 6.4.7]{Oxl92})
that multiplying columns and rows of $\tilde G$ by positive
numbers we can make entries corresponding to edges of
$\mathcal{F}$ to be equal to $\pm 1$. Denote the obtained matrix
by $\widehat G$.

\begin{lemma}\label{L:pm1} If $\tilde G$ satisfies the conditions
of Lemma {\rm \ref{L:det0}}, then $\widehat G$ is a matrix with
entries $-1,0$, and $1$.
\end{lemma}

\begin{proof}{Proof} Assume the contrary, that is, there are
entries $\widehat G_{ij}$ which are not in the set $\{-1,0,1\}$.
Let $\widehat G_{ij}$ be one of such entries satisfying the
additional condition: the level $\ell(G_{ij})$ is the minimal
possible among all entries $\widehat G_{ij}$ which are not in
$\{-1,0,1\}$. Denote by $\widehat{G}(ij)$ the submatrix of
$\widehat G$ which corresponds to $G(ij)$.

Then, by observations preceding Lemma \ref{L:det0}, the expansion
of $\det\widehat{G}(ij)$ contains two non-zero terms: one of them
is $1$ or $-1$, the other is $\widehat{G}_{ij}$ or
$-\widehat{G}_{ij}$. Our assumptions imply that
$\det\widehat{G}(ij)\ne 0$. This contradicts
$\det\tilde{G}(ij)=0$, because $\widehat{G}$ is obtained from
$\tilde G$ using multiplications of columns and rows by numbers.
\end{proof}

In Lemma \ref{L:TU} we show that for functions
$\varphi_\alpha(\ep)$ chosen as above, the matrix $\widehat G$
should be totally unimodular for sufficiently small $\ep$. In
Lemma \ref{L:BM} we show how to estimate the Banach--Mazur
distance between $Z$ and $\mathcal{T}_d$ in the case when
$\widehat G$ is totally unimodular.

\begin{lemma}\label{L:TU} If $\ep>0$ is small enough, the matrix
$\widehat G$ is totally unimodular.
\end{lemma}

\begin{proof}{Proof} The conclusion of Lemma \ref{L:det0} implies
that each entry of $\tilde G$ is a
$\varphi_{d(j_0-1)+1}(\ep)$-approximation of an entry from $G$.
Therefore for small $\ep$ the absolute value of each non-zero
entry of $\tilde G$ is at least $\varphi_{d(j_0-1)}(\ep)/2$. This
implies the following observation.
\medskip

\noindent{\bf Observation.} Each $d\times d$ minor of $\tilde G$
is a product of the corresponding minor of $\widehat G$ and a
number $\zeta$ satisfying
$(\varphi_{d(j_0-1)}(\ep)/2)^d\le\zeta\le 1$.

\begin{proof}{Proof} Consider a square submatrix
$\tilde{\mathcal{S}}$ in $\tilde G$ and the corresponding
submatrix $\widehat{\mathcal{S}}$ in $\widehat G$. If the
corresponding minor is zero, there is nothing to prove. If it is
non-zero, we reorder columns and rows of $\tilde{\mathcal{S}}$ in
such a way that all entries on the diagonal become non-zero, and
do the same reordering with $\widehat{\mathcal{S}}$. Let
$\mathfrak{r}_i,\mathfrak{c}_j>0$ be such that after multiplying
rows of $\widehat{\mathcal{S}}$ by $\mathfrak{r}_i$ and columns of
the resulting matrix by $\mathfrak{c}_j$ we get
$\tilde{\mathcal{S}}$. Then
$$\det(\tilde{\mathcal{S}})=\det(\widehat{\mathcal{S}})\prod_i\mathfrak{r}_i\prod_j\mathfrak{c}_j.$$
On the other hand,
$\mathfrak{r}_i\mathfrak{c}_i\ge\varphi_{d(j_0-1)}(\ep)/2$,
because the diagonal entry of $\widehat{\mathcal{S}}$ is $\pm1$,
and the absolute value of the diagonal entry of
$\tilde{\mathcal{S}}$ is $\ge\varphi_{d(j_0-1)}(\ep)/2$. The
conclusion follows.
\end{proof}

\begin{lemma}\label{L:Gomory} Let $D$ be a $d\times J$ matrix with entries $-1,0$, and $1$, containing a $d\times d$ identity
submatrix. If $D$ is not totally unimodular, then it contains
$(d+2)$ columns $\{\widehat x_1,\dots,\widehat x_{d-2}$, $\widehat
p_1,\widehat p_2,\widehat p_3, \widehat p_4\}$ such that for all
six choices of two vectors from the set $\{\widehat p_1,\widehat
p_2$, $\widehat p_3,\widehat p_4\}$ minors obtained by joining
them to $\{\widehat x_1,\dots,\widehat x_{d-2}\}$ are non-zero.
\end{lemma}

\begin{proof}{Proof} Our argument follows
\cite[pp.~1068--1069]{Cam65} (see, also,
\cite[pp.~269--271]{Sch86}), where a similar statement is
attributed to R.~Gomory.
\medskip

Suppose that $D$ is not totally unimodular, then it has a square
submatrix $\mathcal{S}$ with $|\det(\mathcal{S})|\ge 2$. Let
$\mathcal{S}$ be of size $h\times h$. Reordering columns and rows
of $D$ (if necessary), we may assume that $D$ is of the form:
$$D=\left(\begin{array}{cccc}\mathcal{S} & 0 & I_h & *\\
* & I_{d-h} & 0 & *\end{array}\right),
$$
where $I_h$ and $I_{d-h}$ are identity matrices of sizes $h\times
h$ and $(d-h)\times (d-h)$, respectively, $0$ denote matrices with
zero entries of the corresponding dimensions, and $*$ denote
matrices of the corresponding dimensions with unspecified entries.
\medskip

We consider all matrices which can be obtained from $D$ by a
sequence of the following operations:
\begin{itemize}
\item Addition or subtraction a row to or from another row, \item
Multiplication of a column by $-1$, \end{itemize} provided that
after each such operation we get a matrix with entries $-1,0$, and
$1$.

Among all matrices obtained from $D$ in such a way we select a
matrix $\widehat D$ which satisfies the following conditions:
\begin{itemize}
\item[(1)] Has all unit vectors among its columns; \item[(2)] Has
the maximal possible number $\xi$ of unit vectors among the first
$d$ columns.
\end{itemize}

Observe that $\xi<d$ because the operations listed above preserve
the absolute value of the determinant and at the beginning the
absolute value of the determinant formed by the first $d$ columns
was $\ge 2$. Let $d_r$ be one of the first $d$ columns of
$\widehat D$ which is not a unit vector. Let $\{i_1,\dots,i_t\}$
be indices of its non-zero coordinates. Then at least one of the
unit vectors $e_{i_1},\dots,e_{i_t}$ is not among the first $d$
columns of $\widehat D$ (the first $d$ columns of $\widehat D$ are
linearly independent). Assume that $e_{i_1}$ is not among the
first $d$ columns of $\widehat D$. We can try to transform
$\widehat D$ adding/subtracting the row number $i_1$ to/from rows
number $i_2,\dots, i_t$ (and multiplying the column number $r$ by
$(-1)$, if necessary) into a new matrix $\tilde D$ which satisfies
the following conditions:
\begin{itemize}
\item Has among the first $d$ columns all the unit vectors it had
before; \item Has $e_{i_1}$ as its column number $r$; \item Has
all the unit vectors among its columns.
\end{itemize}
It is not difficult to verify that the only possible obstacle is
that there exists another column $d_t$ in $\widehat D$, such that
for some $s\in\{2,\dots,t\}$
\begin{equation}\label{E:D}\left|\det\left(\begin{array}{cc} D_{i_1r} &
D_{i_1t}\\D_{i_sr} &
D_{i_st}\end{array}\right)\right|=2,\end{equation} where by
$D_{ij}$ we denote entries of $\widehat D$. By the maximality
assumption, a submatrix satisfying (\ref{E:D}) exists.

It is easy to see that letting $\{\widehat p_1,\widehat
p_2,\widehat p_3, \widehat p_4\}=\{d_r,d_s,e_{i_1},e_{i_s}\}$, and
$\{\widehat x_1,\dots,\widehat
x_{d-2}=\{e_1,\dots,e_d\}\backslash\{e_{i_1},e_{i_s}\}$, we get a
set of columns of $\widehat D$ satisfying the required condition.

Since the operations listed above preserve the absolute values of
$d\times d$ minors, the corresponding columns of $D$ form the
desired set.
\end{proof}

\begin{remark}{Remark} Lemma \ref{L:Gomory} can also be obtained by
combining known characterizations of regular and binary matroids,
see \cite{Oxl92} (we mean, first of all, Theorem 9.1.5, Theorem
6.6.3, Corollary 10.1.4, and Proposition 3.2.6).
\end{remark}

We continue our proof of Lemma \ref{L:TU}. Assume the contrary.
Since there are finitely many possible values of $j_0$, there is
$j_0$ and a subset $\Phi_4\subset(0,1)$, whose closure contains
$0$, satisfying the condition:

For each $\ep\in\Phi_4$ there is $Z\in\mathcal{Z}_\ep$ such that
following the construction, we get the preselected value of $j_0$,
and the obtained matrix $\widehat G$ is not totally unimodular.
\medskip

Since the entries of $\widehat G$ are integers, the absolute
values of the minors are at least one. We are going to show that
the corresponding minors of $G$ are also `sufficiently large', and
get a contradiction using Lemma \ref{L:six}.
\medskip

By the observation above the corresponding minors of $\tilde G$
are at least $(\varphi_{d(j_0-1)}(\ep)/2)^d$. The Euclidean norm
of a column in $\tilde G$ is at most
$1+d\varphi_{d(j_0-1)+1}(\ep)$. Applying Lemma \ref{L:detperturb}
$d$ times we get that the corresponding minor of $G$ are at least
$$(\varphi_{d(j_0-1)}(\ep)/2)^d-d^2\varphi_{d(j_0-1)+1}(\ep)\cdot(1+d\varphi_{d(j_0-1)+1}(\ep))^{d-1}.$$

We are going to use Lemma \ref{L:six} for $x_1,\dots,x_{d-2},
p_1,p_2,p_3,p_4$ defined in the following way. Let  $\check
x_1,\dots,\check x_{d-2},\check p_1,\check p_2,\check p_3,\check
p_4$ be the columns of $G$ corresponding to the columns $\widehat
x_1,\dots,\widehat x_{d-2}, \widehat p_1,\widehat p_2,\widehat
p_3, \widehat p_4$ of $\widehat G$, and $x_1,\dots,x_{d-2},
p_1,p_2,p_3,p_4$ be their normalizations (that is, $x_1=\check
x_1/||\check x_1||$, etc). Since norms of columns of $G$ are $\le
1$, the condition (\ref{E:detge}) of Lemma \ref{L:six} is
satisfied with
$$\chi(\ep)=(\varphi_{d(j_0-1)}(\ep)/2)^d-d^2\varphi_{d(j_0-1)+1}(\ep)\cdot(1+d\varphi_{d(j_0-1)+1}(\ep))^{d-1}.$$

Now we recall that columns $\{g_j\}$ of $G$ satisfy (\ref{E:gx})
for some vectors $\mathfrak{b}_j\in\co(\omega(\ep), \delta(\ep))$.
Hence the distance from $x_1,\dots,x_{d-2}, p_1,p_2,p_3,p_4$ to
the corresponding vectors $\mathfrak{b}_j$ is $\le
2d\varphi_{dj_0}(\ep)$. By (\ref{E:Obs2}) the condition
(\ref{E:measure}) is satisfied with
$$\pi(\ep)=2d\varphi_{dj_0}(\ep)+C_5(d)\omega(\ep)$$
and
$$\sigma(\ep)=c_3(d)\delta(\ep).$$
The fact that the conditions {\bf (2)} and {\bf (3)} of Lemma
\ref{L:six} are satisfied is verified in the same way as at the
end of Lemma \ref{L:det0}, the only difference is that instead of
(\ref{E:Ochi}) we have $(\varphi_{d(j_0-1)}(\ep))^d=O(\chi(\ep))$.
This does not affect the rest of the argument. Therefore, under
the same condition on $k$ as in Lemma \ref{L:det0} we get, by
Lemma \ref{L:six}, that $\widehat G$ should be totally unimodular
if $\ep>0$ is small enough.
\end{proof}

\begin{lemma}\label{L:BM} If $\widehat G$ is totally unimodular,
then there exists a zonotope $T\in\mathcal{T}_d$ such that
$$d(Z,T)\le\mathfrak{t}_d(\ep),$$
where $\mathfrak{t}_d(\ep)$ is a function satisfying
$\lim_{\ep\downarrow 0}\mathfrak{t}_d(\ep)=1$.
\end{lemma}

\begin{proof}{Proof} Observe that the matrix $\tilde G$ can be
obtained from $\widehat G$ using multiplications of rows and
columns by positive numbers. Hence, re-scaling the basis
$\{e_i\}$, if necessary, we get: columns of $\tilde G$ with
respect to the re-scaled basis are of the form $a_i\tau_i$, where
$\tau_i$ are columns of a totally unimodular matrix (see the
definition of $\mathcal{T}_d$ in the introduction).

We are going to approximate the measure $\cm$ by a measure
$\widehat{\mu}$ supported on vectors which are normalized columns
of $\tilde G$. Recall that $\cm$ is supported on a finite subset
of $\check S$.

The approximation is constructed in the following way. We erase
the measure $\cm$ supported outside
$(\co(\omega(\ep),\delta(\ep)))_{C_3(d)\omega(\ep)}$. The total
mass of the measure erased in this way is small by (\ref{E:Obs1}).
As for the measure supported on
$\mathcal{B}:=(\co(\omega(\ep),\delta(\ep)))_{C_3(d)\omega(\ep)}$,
we approximate each atom of it by the atom of the same mass
supported on the nearest normalized column of $\tilde G$. We
denote the nearest to $z\in\supp\check\mu$ normalized column of
$\tilde G$ by $\mathcal{A}(z)$. If there are several such columns,
we choose one of them.\medskip

Now we estimate the distance from a point of
$(\co(\omega(\ep),\delta(\ep)))_{C_3(d)\omega(\ep)}$ to the
nearest normalized column of $\tilde G$. The distance from this
point to $\co(\omega(\ep),\delta(\ep))$ is ${C_3(d)\omega(\ep)}$,
the distance from a point from $\co(\omega(\ep),\delta(\ep))$ to
the point from $\Theta(\omega(\ep),\delta(\ep))$ with the same top
set (or its opposite), by Lemma \ref{L:smalldistance}, can be
estimated from above by
$\sqrt{2\frac{\nu(\ep)}{(\rho(\ep))^2}+4d\rho(\ep)^2}$. The
distance from a point in $\Theta(\omega(\ep),\delta(\ep))$ to the
corresponding column of $G$ is estimated in (\ref{E:gx}), it is
$\le d\cdot\varphi_{dj_0}(\ep)$, so it is $\le
d\cdot\varphi_{1}(\ep)$, and the distance from a column of $G$ to
the corresponding column of $\tilde G$ is $\le
d\cdot\varphi_{d(j_0-1)+1}(\ep)\le d\cdot\varphi_1(\ep)$. Since we
have to normalize this vector, the total distance from a point of
$(\co(\omega(\ep),\delta(\ep)))_{C_3(d)\omega(\ep)}$ to the
nearest normalized column of $\tilde G$ can be estimated from
above by
$${C_3(d)\omega(\ep)}+\sqrt{2\frac{\nu(\ep)}{(\rho(\ep))^2}+4d\rho(\ep)^2}+4
d\cdot\varphi_{1}(\ep)$$ It is clear that this function, let us
denote it by $\zeta(\ep)$, tends to $0$ as $\ep\downarrow 0$,
recall that $\rho(\ep)=e^k$, $\nu(\ep)=\ep^{3k}$,
$\omega(\ep)=\ep^{4k}$,
$\displaystyle{\varphi_1(\ep)=\ep^{{\left(\frac1{d+1}\right)^{d\mathfrak{n}-1}}}}$.
The obtained measure corresponds to a zonotope from
$\mathcal{T}_d$. Let us denote this zonotope by $T$.
\medskip

Since the dual norms to the gauge functions of $Z$ and $T$ are
their support functions, we get the estimate
$$d(T,Z)\le\sup_{u\in\check S}\frac{\check h_Z(u)}{\check h_T(u)}\cdot
\sup_{u\in\check S}\frac{\check h_T(u)}{\check h_Z(u)}.$$

So it is enough to show that
\begin{equation}\label{E:C1C2}
C_1(d,\ep)\le \frac{\check h_T(u)}{\check h_Z(u)}\le
C_2(d,\ep),\end{equation} where $\lim_{\ep\downarrow
0}C_1(d,\ep)=\lim_{\ep\downarrow 0}C_2(d,\ep)=1$. \medskip

Observe that Lemma \ref{L:bases} implies that there exists a
constant $0<C_7(d)<\infty$ such that
\begin{equation}\label{E:C7}C_7(d)\le \check h_Z(u),~~ \forall
u\in\check S.\end{equation}

We have
\begin{align*}\begin{split} \check h_Z(u)&=\int_{\check S}|\langle
u,z\rangle|d\cm(z)\le \int_{\check S\backslash
\mathcal{B}}|\langle
u,z\rangle|d\cm(z)\\
&\quad+ \int_{\check S} |\langle
u,z\rangle|d\widehat{\mu}(z)+\sum_{z\in\supp\check \mu\cap
\mathcal{B}}\left(|\langle u,z\rangle-\langle
u,\mathcal{A}(z)\rangle|\right)\check\mu(z)
\\&
\le C_4(d)\frac{\delta(\ep)}{\omega^{d-1}(\ep)}+\check
h_T(u)+\zeta(\ep)\cm(\check S), ~~ \forall u\in\check S.
\end{split}\end{align*}
In a similar way we get
\begin{align*}\begin{split} \check h_T(u)&=\int_{\check S}|\langle
u,z\rangle|d\widehat{\mu}(z)\le \int_{\mathcal{B}} |\langle
u,z\rangle|d\cm(z)\\
&\quad+\sum_{z\in\supp\check \mu\cap\mathcal{B}}\left(|\langle
u,z\rangle-\langle u,\mathcal{A}(z)\rangle|\right)\check\mu(z)
\\&
\le \check h_Z(u)+\zeta(\ep)\cm(\check S), ~~ \forall u\in S.
\end{split}\end{align*}

Using (\ref{E:C7}) we get
$$1-\frac{C_4(d)\frac{\delta(\ep)}{\omega^{d-1}(\ep)}}{C_7(d)}-\frac{\zeta(\ep)\cm(\check
S)}{C_7(d)}\le\frac{\check h_T(u)}{\check h_Z(u)}\le
1+\frac{\zeta(\ep)\cm(\check S)}{C_7(d)}.$$ It is an estimate of
the form (\ref{E:C1C2}), Q.E.D.
\end{proof}

It is clear that Lemma \ref{L:BM} completes our proof of Lemma
\ref{L:APPROX}.
\end{proof}

\section{Proof of Theorem \ref{T:NP}}

\begin{proof}{Proof} We start by proving Theorem \ref{T:NP} for polyhedral
$X$. In this case we can consider $X$ as a subspace of
$\ell_\infty^m$ for some $m\in {\bf N}$. Since $X$ has an MVSE
which is not a parallelepiped, there exists a linear projection
$P:\ell_\infty^m \to X$ such that $P(B_\infty^m)$ has the minimal
possible volume, but $P(B_\infty^m)$ is not a parallelepiped. Let
$d=\dim X$, let $\{q_1,\dots,q_{m-d}\}$ be an orthonormal basis in
$\ker P$ and let $\{\tilde q_1,\dots,\tilde q_d\}$ be an
orthonormal basis in the orthogonal complement of $\ker P$. As it
was shown in Lemma \ref{L:shape}, $P(B_\infty^m)$ is linearly
equivalent to the zonotope spanned by rows of $\tilde Q=[\tilde
q_1,\dots,\tilde q_d]$. By the assumption this zonotope is not a
parallelepiped. It is easy to see that this assumption is
equivalent to: there exists a minimal linearly dependent
collection of rows of $\tilde Q$ containing $\ge 3$ rows. This
condition implies that we can reorder the coordinates in
$\ell_\infty^m$ and multiply the matrix $\tilde Q$ from the right
by an invertible $d\times d$ matrix $C_1$ in such a way that
$\tilde QC_1$ has a submatrix of the form
$$\left(
\begin{array}{cccc}
1 & 0 & \dots & 0\\
0 & 1 & \dots & 0\\
\vdots & \vdots & \ddots & \vdots\\
0 & 0 & \dots & 1\\
a_1 & a_2 & \dots & a_d
\end{array}\right),
$$
where $a_1\ne 0$ and $a_2\ne 0$. Let $\mathcal{X}$ be a matrix
whose columns form a basis of $X$. The argument of \cite{laa} (see
the conditions (1)--(3) on p.~96) implies that $\mathcal{X}$ can
be multiplied from the right by an invertible $d\times d$ matrix
$C_2$ in such a way that $\mathcal{X}C_2$ is of the form
$$\left(
\begin{array}{cccc}
1 & 0 & \dots & 0\\
0 & 1 & \dots & 0\\
\vdots & \vdots & \ddots & \vdots\\
0 & 0 & \dots & 1\\
\sign a_1 & \sign a_2 & \dots & *\\
\vdots & \vdots & \ddots & \vdots
\end{array}\right),
$$
where at the top there is an $d\times d$ identity matrix, and all
minors of the matrix $\mathcal{X}C_2$ have absolute values $\le
1$.
\medskip

Changing signs of the first two columns, if necessary, we get that
the subspace $X\subset\ell_\infty^m$ is spanned by columns of the
matrix
\begin{equation}\label{Z}
\left(
\begin{array}{ccccc}
\pm1 & 0 & 0 &\dots & 0\\
0 & \pm1 & 0 &\dots & 0\\
0 & 0 & 1 &\dots & 0\\
\vdots & \vdots & \vdots & \ddots & \vdots\\
0 & 0 & 0 &\dots & 1\\
1 & 1 & * & \dots & *\\
b_1 & c_1 & * & \dots & *\\
b_2 & c_2 & * & \dots & *\\
\vdots & \vdots & \vdots & \ddots & \vdots\\
b_{m-l-1} & c_{m-l-1} & * & \dots & *
\end{array}\right).
\end{equation}
The condition on the minors implies that $|b_i|\le 1$, $|c_i|\le
1$, and $|b_i-c_i|\le 1$ for each $i$. Therefore the subspace,
spanned in $\ell_\infty^m$ by the first two columns of the matrix
(\ref{Z}) is isometric to ${\bf R}^2$ with the norm
$$||(\alpha, \beta)||=\max(|\alpha|, |\beta|, |\alpha+\beta|).$$
It is easy to see that the unit ball of this space is linearly
equivalent to a regular hexagon. Thus, Theorem \ref{T:NP} is
proved in the case when $X$ is polyhedral.
\medskip

Proving the result for general, not necessarily polyhedral, space,
we shall denote the space by $Y$. We use Theorem \ref{T:MVSE}.
Actually we need only the following corollary of it: {\it  Each
MVSE is a polyhedron.} Therefore we can apply the following result
to each MVSE.

\begin{lemma}\label{T:polyhedral} {\rm\cite[Lemma 1]{Ost04}}
Let $Y$ be a finite dimensional space and let $A$ be a polyhedral
MVSE for $Y$. Then there exists another norm on $Y$ such that the
obtained normed space $X$ satisfies the conditions:
\medskip

\noindent{\rm (1)} $X$ is polyhedral;

\noindent{\rm (2)} $B_X\supset B_Y$;

\noindent{\rm (3)} $A$ is an MVSE for $X$.
\end{lemma}

So we consider the space $Y$ as being embedded into a polyhedral
space $X$ with the embedding satisfying the conditions of Lemma
\ref{T:polyhedral}. By the first part of the proof the space $X$
satisfies the conditions of Theorem \ref{T:NP} and we may assume
that $X$ is a subspace $\ell_\infty^m$ in the way described in the
first part of the proof. So $X$ is spanned by columns - let us
denote them by $e_1, \dots, e_d$ - of the matrix (\ref{Z}) in
$\ell_\infty^m$. It is easy to see that to finish the proof it is
enough to show that the vectors $e_1$, $e_2$, $e_1-e_2$ are in
$B_Y$.
\medskip

It turns out each of these points is the center of a facet of a
minimum-volume parallelepiped containing $B_X$. In fact, let
$\{f_i\}_{i=1}^m$ be the unit vector basis of $\ell_\infty^m$. Let
$P_1$ and $P_2$ be the projections onto $Y$ with the kernels
$\lin\{f_{d+1},\dots,f_m\}$ and $\lin\{f_1,f_{d+2},\dots,f_m\}$,
respectively (recall that $Y$, as a linear space, coincides with
$X$). The analysis from \cite[pp.~318--319]{jfa} shows that
$P_1(B_\infty^m)$ and $P_2(B_\infty^m)$ have the minimal possible
volume among all linear projections of $B_\infty^m$ into $X$. It
is easy to see that $P_1(B_\infty^m)$ and $P_2(B_\infty^m)$ are
parallelepipeds.
\medskip

We show that $e_1$, $e_2$ are centers of facets of
$P_1(B_\infty^m)$, and that $e_1-e_2$ is the center of a facet of
$P_2(B_\infty^m)$. In fact, the centers of facets of
$P_1(B_\infty^m)$ coincide with $P_1(f_1),\dots,P_1(f_d)$, and it
is easy to check that $P_1(f_i)=e_i$ for $i=1,\dots,d$. As for
$P_2$, we observe that
$e_1-e_2\in\lin\{f_1,f_2,f_{d+2},\dots,f_m\}$, and the coefficient
near $f_2$ in the expansion of $e_1-e_2$ is $\pm 1$. Therefore
$P_2(f_2)=\pm(e_1-e_2)$.
\medskip

Since the projections $P_1$ and $P_2$ satisfy the minimality
condition from \cite[Lemma 1]{laa} (see, also
\cite[pp.~318--319]{jfa}), the parallelepipeds $P_1(B_\infty^m)$
and $P_2(B_\infty^m)$ are MVSE for $X$. Hence, by the conditions
of Lemma \ref{T:polyhedral}, they are MVSE for $Y$ also. Hence,
they are minimum-volume parallelepipeds containing $B_Y$. On the
other hand, it is known, see  \cite[Lemma 3$\cdot$1]{PS}, that
centers of facets of minimal-volume parallelepipeds containing
$B_Y$ should belong to $B_Y$, we get $e_1,e_2,e_1-e_2\in B_Y$. The
theorem follows.
\end{proof}

I would like to thank Gideon Schechtman for turning my attention
to the fact that the class $\mt_d$ was studied in works on lattice
tiles.

\end{large}

\end{document}